\documentclass[12pt]{article}
\usepackage{amsmath, amsthm, amssymb}
\usepackage{hyperref}
\usepackage[margin=1cm]{caption}
\usepackage{verbatim}
\usepackage[top=1.0in, bottom=1.0in, left=1.0in, right=1.0in]{geometry}
\usepackage{graphicx}

\usepackage{tkz-graph}
\usetikzlibrary{arrows}
\usetikzlibrary{shapes}
\usepackage[position=bottom]{subfig}

\usepackage{longtable}
\usepackage{array}

\usepackage{sectsty}
\allsectionsfont{\sffamily}

\theoremstyle{plain}
\newtheorem{thm}{Theorem}[section]

\newtheorem{lem}[thm]{Lemma}
\newtheorem{conjecture}[thm]{Conjecture}
\newtheorem{cor}[thm]{Corollary}

\newtheorem{observation}{Observation}
\newtheorem{obs}[observation]{Observation}

\newtheorem{problem}{Problem}
\newtheorem{clm}{Claim}
\newtheorem*{VAL}{Vizing's Adjacency Lemma (VAL)}

\theoremstyle{definition}
\newtheorem{defn}{Definition}
\theoremstyle{remark}

\newcommand{\fancy}[1]{\mathcal{#1}}

\newcommand{\IN}{\mathbb{N}}

\newcommand{\set}[1]{\left\{ #1 \right\}}
\newcommand{\setb}[3]{\left\{ #1 \in #2 \mid #3 \right\}}
\newcommand{\setbs}[2]{\left\{ #1 \mid #2 \right\}}
\newcommand{\card}[1]{\left|#1\right|}
\newcommand{\size}[1]{\left\Vert#1\right\Vert}

\newcommand{\floor}[1]{\left\lfloor#1\right\rfloor}
\newcommand{\func}[3]{#1\colon #2 \rightarrow #3}

\newcommand{\irange}[1]{\left[#1\right]}

\newcommand{\parens}[1]{\left( #1 \right)}
\newcommand{\brackets}[1]{\left[ #1 \right]}

\newcommand{\DefinedAs}{\mathrel{\mathop:}=}
\newcommand{\pot}{\operatorname{pot}}

\def\adj{\leftrightarrow}

\def\H{\fancy{H}}

\newcommand\numberthis{\addtocounter{equation}{1}\tag{\theequation}}

\def\adj{\leftrightarrow}
\def\ch{\textrm{ch}}

\begin{document}
\title{Edge-coloring via fixable subgraphs}
\author{Daniel W. Cranston\thanks{Department of Mathematics and Applied
Mathematics, Viriginia Commonwealth University, Richmond, VA;
\texttt{dcranston@vcu.edu}; 
Research of the first author is partially supported by NSA Grant
98230-15-1-0013.}
\and
Landon Rabern\thanks{LBD Data Solutions, Lancaster, PA;
\texttt{landon.rabern@gmail.com}}
}
\maketitle

\begin{abstract}
Many graph coloring proofs proceed by showing that a minimal counterexample to
the theorem being proved cannot contain certain configurations, and then
showing that each graph under consideration contains at least one such
configuration; these configurations are called \emph{reducible} for that theorem.
(A \emph{configuration} is a subgraph $H$, along with specified degrees
$d_G(v)$ in the original graph $G$ for each vertex of $H$.)

We give a general framework for showing that configurations are reducible for
edge-coloring.  A particular form of reducibility, called \emph{fixability},
can be considered without reference to a containing graph.  This has two key
benefits: (i) we can now formulate necessary conditions for fixability, and
(ii) the problem of fixability is easy for a computer to solve. The necessary
condition of \emph{superabundance} is sufficient for multistars and we
conjecture that it is sufficient for trees as well (this would generalize the
powerful technique of Tashkinov trees). 

Via computer, we can generate thousands of reducible
configurations, but we have short proofs for only a small fraction of these. 
The computer can write \LaTeX\ code for its proofs, but they are only
marginally enlightening and can run thousands of pages long.  We give examples
of how to use some of these reducible configurations to prove conjectures on
edge-coloring for small maximum degree.  Our aims in writing this paper
are 
(i) to provide a common context for a variety of reducible configurations for
edge-coloring and 
(ii) to spur development of methods for humans to understand what the computer
already knows.
\end{abstract}

\section{Introduction}
Suppose we want to $k$-color a graph $G$. If we already have a $k$-coloring of
an induced subgraph $H$ of $G$, we might try to extend this coloring to all of
$G$.  We can view this task as coloring $G-H$ from lists (this is called
\emph{list-coloring}), where each vertex $v$ in $G-H$ gets a list of colors
formed from $\set{1, \ldots, k}$ by removing all colors used on its
neighborhood in $H$.  
Often we cannot complete just any $k$-coloring of $H$ to all of
$G$.  Instead, we may need to modify the $k$-coloring of $H$ to get a coloring
we can extend.  Given rules for how we may modify the $k$-coloring
of $H$, we can view our original problem
as the problem of list coloring $G-H$, where each
vertex gets a list as before, but now we can modify these lists in certain ways.

As an example of this approach, the second author proved \cite{HallGame} 
a common generalization of Hall's marriage theorem and Vizing's theorem on
edge-coloring.  The present paper generalizes a special case of this result and
puts it into a broader context.  
Since we often want to prove coloring results for all graphs having
certain properties, and not just some fixed graph, we only have partial control
over the outcome of a recoloring of $H$. For example, if we swap colors red and
green in a component $C$ of the red-green subgraph (that is, we perform a Kempe
change), we may succeed in making some desired vertex red,
but if $C$ is somewhat arbitrary, we cannot precisely control what happens
to the colors of its other vertices.  
We model this lack of control as a two-player game---we move by 
recoloring a vertex as we desire and then the other player gets a turn to muck things up. 
In the original context, where we want to color $G$, our opponent is the graph
$G$; more precisely, the embedding of $G-H$ in $G$ is one way to describe a
strategy for the second player. The general paradigm that we described above is
for vertex coloring.  In the rest of the paper, we consider only the special
case that is edge-coloring (or, equivalently, vertex coloring line graphs).  

All of our multigraphs are loopless.  Let $G$ be a multigraph, $L$ a list
assignment on $V(G)$, and $\pot(L) = \bigcup_{v\in V(G)} L(v)$. An
\emph{$L$-pot} is a set $X$ containing $\pot(L)$.  
We typically let $P$ denote an arbitrary $L$-pot.  
An \emph{$L$-edge-coloring}
is an edge-coloring $\pi$ of $G$ such that $\pi(xy) \in L(x) \cap L(y)$ for all
$xy \in E(G)$;  
furthermore, we require $\pi(xy)\ne \pi(xz)$ for each vertex $x$ and distinct
neighbors $y$ and $z$ of $x$.
For the maximum degree in $G$ we write $\Delta(G)$, or simply
$\Delta$, when $G$ is clear from context.
For the edge-chromatic number of $G$ we write $\chi'(G)$.
We often denote the set $\{1,\ldots,k\}$ by $[k]$.

\section{Completing edge-colorings}
Our goal is to convert a partial $k$-edge-coloring of a multigraph $M$ into a
$k$-edge-coloring of (all of) $M$.  For a partial $k$-edge-coloring $\pi$ of
$M$, let $M_\pi$ be the subgraph of $M$ induced by the uncolored edges and let
$L_\pi$ be the list assignment on the vertices of $M_\pi$ given by 
$L_\pi(v) = \irange{k} - \setbs{\tau}{\pi(vx) = \tau \text{ for some edge  } vx \in E(M)}$. 

Kempe chains give a powerful technique for converting a partial
$k$-edge-coloring into a $k$-edge-coloring of the whole graph.  The idea is to
repeatedly exchange colors on two-colored paths until the uncolored subgraph
$M_\pi$ has an edge-coloring $\zeta$ from its lists, that is, such that
$\zeta(xy) \in L_\zeta(x) \cap
L_\zeta(y)$ for all $xy \in E(M_\pi)$.  (One advantage of considerng the special
case that is edge-coloring is that every Kempe chain is either a path or an even
cycle.) In this sense the original list
assignment $L_\pi$ on $M_\pi$ is \emph{fixable}. In the next section, we give
an abstract definition of this notion that frees us from the embedding in the
containing graph $M$.  As we will see, computers enjoy this new freedom.


\subsection{Fixable graphs}
Thinking in terms of a two-player game is a good aid to intuition and we
encourage the reader to continue doing so. However, a simple recursive
definition is equivalent and has far less baggage. For distinct colors $a,b
\in P$, let $S_{L,a,b}$ be all the vertices of $G$ that have exactly one of $a$
or $b$ in their list; more precisely, $S_{L,a,b} =
\setb{v}{V(G)}{\,\card{\set{a,b} \cap L(v)} = 1}$.  

\begin{defn}
$G$ is \emph{$(L, P)$-fixable} if either
\begin{enumerate}
\item[(1)] $G$ has an $L$-edge-coloring; or
\item[(2)] there are different colors $a,b \in P$ such that for every partition
$X_1, \ldots, X_t$ of $S_{L,a,b}$ into sets of size at most two, there exists $J
\subseteq \irange{t}$ so that $G$ is $(L', P)$-fixable, where $L'$ is formed
from $L$ by swapping $a$ and $b$ in $L(v)$ for every $v \in \bigcup_{i \in J} X_i$.
\end{enumerate}
\end{defn}

The meaning of (1) is clear.  Intuitively, (2) says the following.  There is
some pair of colors, $a$ and $b$, such that regardless of how the vertices of
$S_{L,a,b}$ are paired via Kempe chains for colors $a$ and $b$ (or not paired
with any vertex of $S_{L,a,b}$), we can swap the colors on some subset $J$ of
the Kempe chains so that the resulting partial edge-coloring is fixable.

We write $L$-fixable as shorthand for $(L, \pot(L))$-fixable. When $G$ is $(L,
P)$-fixable, the choices of $a,b$, and $J$ in each application of (2) determine
a tree where all leaves have lists satisfying (1).  The \emph{height} of $(L,
P)$ is the minimum possible height of such a tree.  We write $h_G(L, P)$ for
this height and let $h_G(L, P) = \infty$ when $G$ is not $(L,P)$-fixable. 

\begin{lem}\label{FixableCompletesColoring}
If a multigraph $M$ has a partial $k$-edge-coloring $\pi$ such that $M_\pi$ is $(L_\pi, \irange{k})$-fixable, then $M$ is $k$-edge-colorable.
\end{lem}
\begin{proof}
Our proof is by induction on the height of $(L_\pi,[k])$.
Choose a partial $k$-edge-coloring $\pi$ of $M$ such that $M_\pi$ is $(L_\pi,
\irange{k})$-fixable. 
If $h_{M_\pi}\parens{L_\pi, \irange{k}} = 0$, then (1) must hold for $M_\pi$
and $L_\pi$; that is, $M_\pi$ has an edge-coloring $\zeta$ such that $\zeta(x)
\in L_\pi(x) \cap L_\pi(y)$ for all $xy \in E(M_\pi)$.  Now 
$\pi \cup \zeta$ is the desired $k$-edge-coloring of $M$.  

So we may assume that $h_{M_\pi}\parens{L_\pi, \irange{k}} > 0$.  Choose colors
$a,b \in \irange{k}$ to satisfy (2) and give a tree of height
$h_{M_\pi}\parens{L_\pi, \irange{k}}$.  Let $H$ be the subgraph of $M$ induced
on all edges colored $a$ or $b$,
and let $S$ be the vertices in $M_\pi$ with degree exactly one in $H$.  
For each $x \in S$, let $C_x$ be the component of $H$ containing $x$.
Since $\card{V(C_x) \cap S} \in \set{1,2}$, the components of $H$ give a
partition $X_1, \ldots, X_t$ of $S$ into sets of size
at most two.  Further, exchanging colors $a$ and $b$ on $C_x$ has the effect
of swapping $a$ and $b$ in $L_\pi(v)$ for each $v \in V(C_x) \cap S$.  So we
achieve the needed swapping of colors in the lists in (2) by exchanging
colors on the components of $H$.  

By (2) there is $J \subseteq \irange{t}$ so
that $M_\pi$ is $(L', \irange{k})$-fixable, where $L'$ is formed from $L_\pi$ by
swapping $a$ and $b$ in $L_\pi(v)$ for every $v \in \bigcup_{i \in J} X_i$. 
In fact, there is a $J$ such that $(L',[k])$ has height less than that of
$(L,[k])$.
Let $\pi'$ be the partial $k$-edge-coloring of $M$ created from
$\pi$ by performing the color exchanges to create $L'$ from $L_\pi$.  
By the induction hypothesis, $M$ is $k$-edge-colorable.
\end{proof}

\subsection{Some examples}
A graph $G$ is \emph{$\Delta$-edge-critical}, or simply \emph{edge-critical},
if $\chi'(G)>\Delta$, but $\chi'(G-e)\le\Delta$ for every edge $e$.  
A \emph{configuration} is a subgraph $H$, along with specified degrees $d_G(v)$
in the original graph for each vertex of $H$. A configuration $H$ is
\emph{reducible} if there exists an edge $e\in E(H)$ such
that whenever $H$ appears as a subgraph (not necessarily induced) of a graph
$G$, if $G-e$ has a $\Delta$-edge-coloring, then so does $G$.
A central tool for proving reducibility for edge-coloring is Vizing's Adjacency
Lemma.  For example, it yields a short proof of Vizing's Theorem that
$\chi'(G)\le \Delta+1$ for every simple graph $G$.

\begin{VAL}
Let $G$ be a $\Delta$-critical graph.  If $xy\in E(G)$, then $x$ is adjacent to
at least $\max\{2,\Delta-d(y)+1\}$ vertices of degree $\Delta$.
\end{VAL}

We can view VAL as giving conditions for the degrees of a vertex and its
neighbors that yield a reducible configuration.  Our goal now is to prove
similar statements for larger configurations; we'd like a way to talk about
configurations being reducible for $k$-edge-coloring. 
Lemma \ref{FixableCompletesColoring} gives us this with respect to a fixed
partial $k$-edge-coloring $\pi$, but we want a condition independent of the
particular coloring.  Note that we have a lower bound on the sizes of the lists
in $L_\pi$; specifically, if $\pi$ is a partial $k$-edge-coloring of a
multigraph $M$, then $|L_{\pi}(v)| \ge k + d_{M_\pi}(v) - d_M(v)$ for every $v
\in M_{\pi}$. 
This observation motivates the following defintion.

\begin{defn}
If $G$ is a graph and $\func{f}{V(G)}{\IN}$, then $G$ is \emph{$(f,k)$-fixable} 
if $G$ is $(L, \irange{k})$-fixable for every $L$ with $|L(v)| \ge k + d_{G}(v)
- f(v)$ for all $v \in V(G)$.
\end{defn}

This definition enables us to state our desired condition on reducible configurations for
$k$-edge-coloring, which follows directly from Lemma \ref{FixableCompletesColoring}. 
\begin{obs}
If $G$ is $(f,k)$-fixable, then $G$ cannot be a subgraph of a
$(k+1)$-edge-critical graph $M$ where $d_M(v) \le f(v)$ for all $v \in V(G)$.  
\end{obs}
Now we can talk about a graph $G$ with vertices labeled by $f$ being
$k$-fixable.  The computer is extremely good at finding $k$-fixable
graphs.  Combined with discharging arguments\footnote{The discharging method is
a counting technique commonly used in coloring proofs to show that the graph
under consideration must contain a reducible configuration. For an introduction
to this method,
see~\cite{discharging13}.}, this gives a powerful method for proving (modulo
trusting the computer) edge-coloring results for small $\Delta$.  We'll see
some examples of such proofs later; for now Figure \ref{fig:small3} shows some
$3$-fixable graphs.  A gallery of hundreds more fixable graphs is available at
\url{https://dl.dropboxusercontent.com/u/8609833/Web/GraphData/Fixable/index.html}.

	\begin{figure}[htb]
		\includegraphics[scale=0.25]{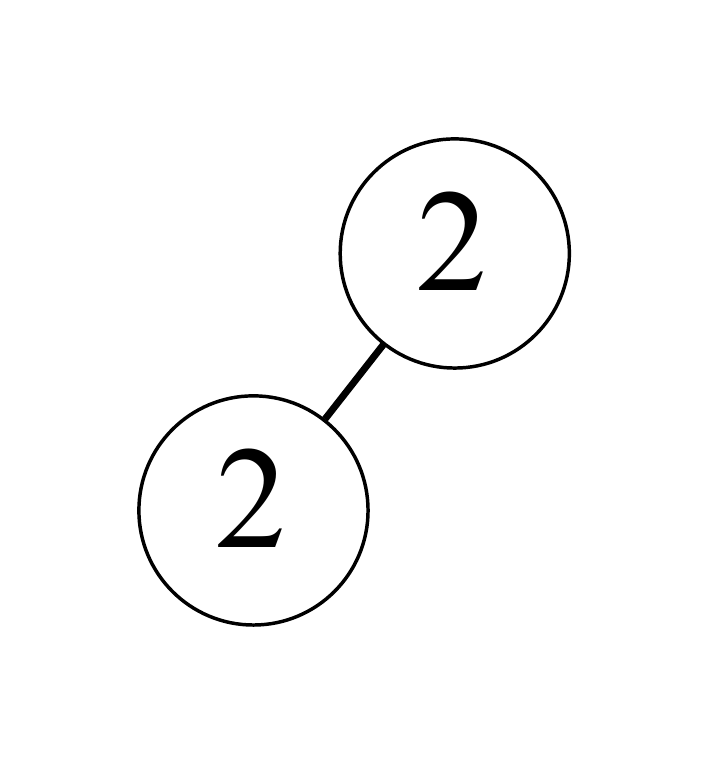}
		\includegraphics[scale=0.25]{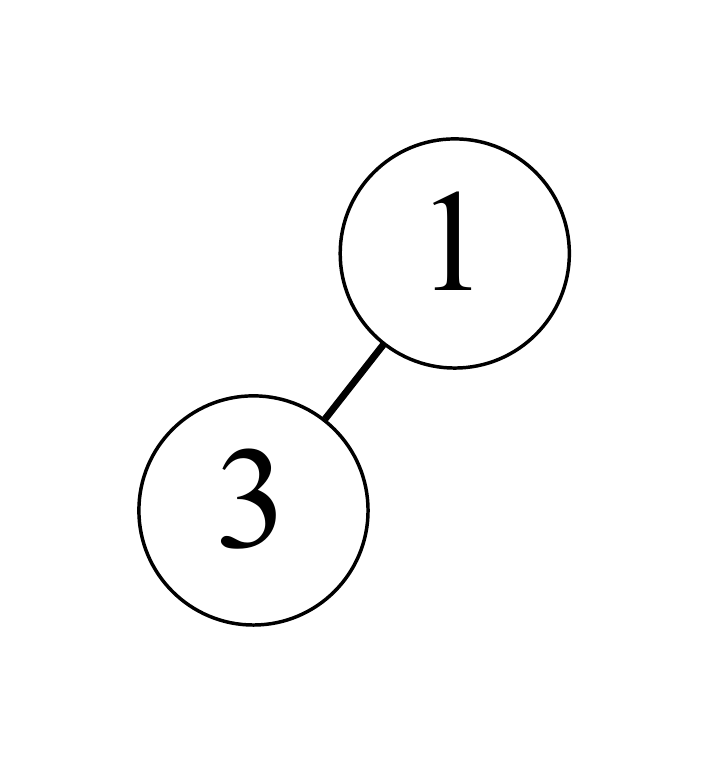}
		\includegraphics[scale=0.25]{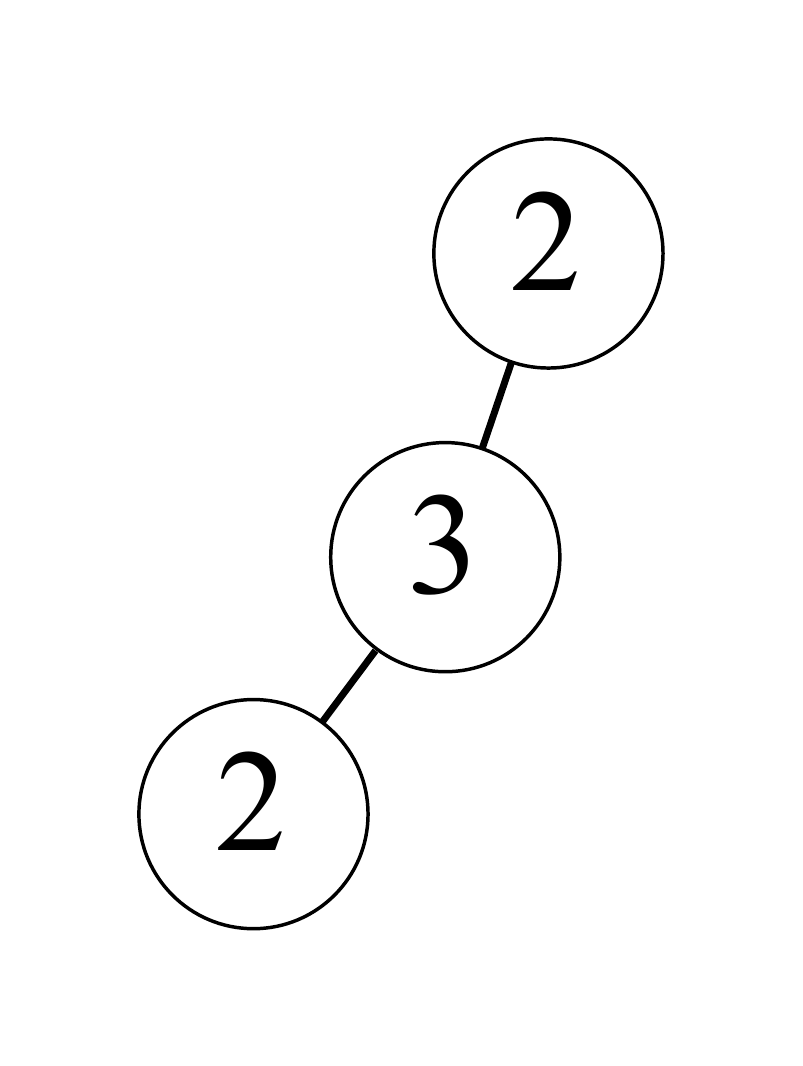}
		\includegraphics[scale=0.25]{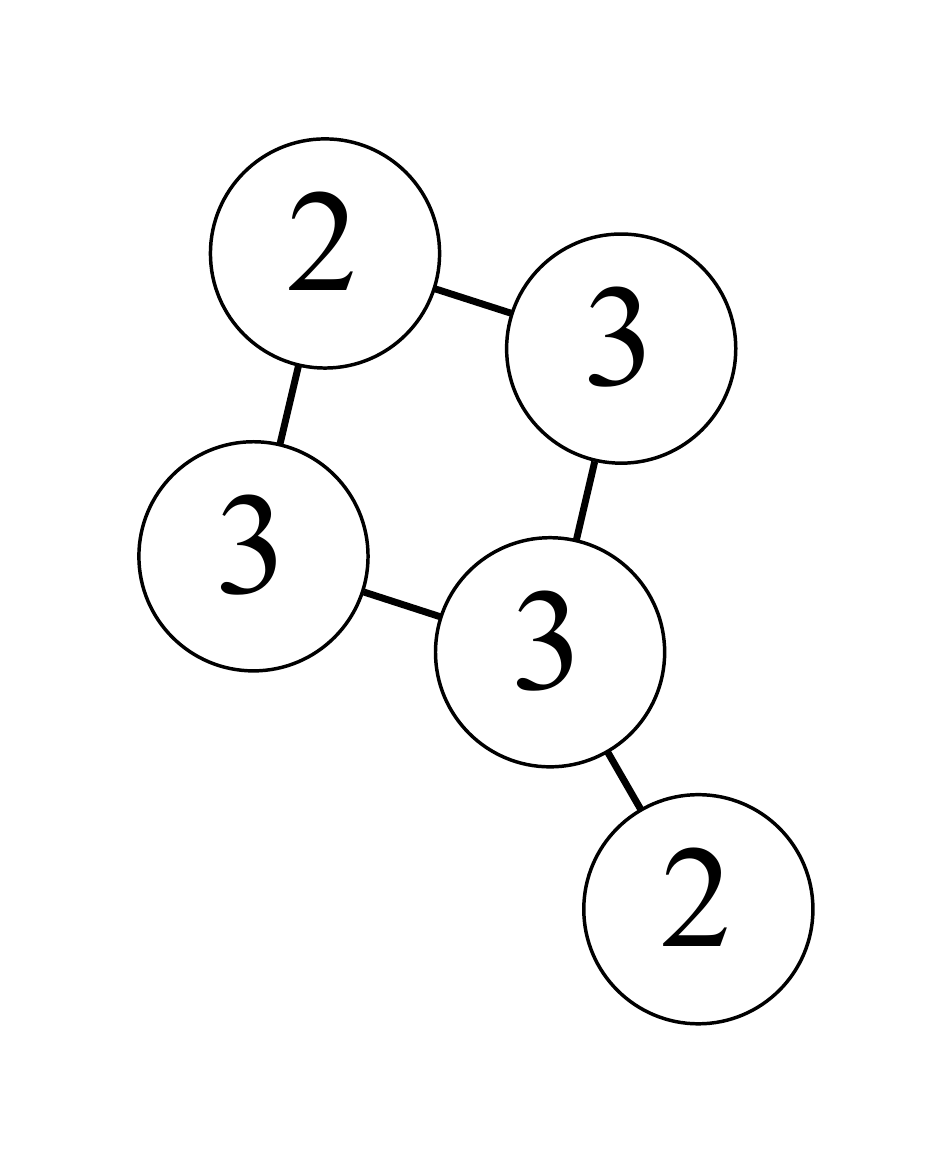}
		\includegraphics[scale=0.25]{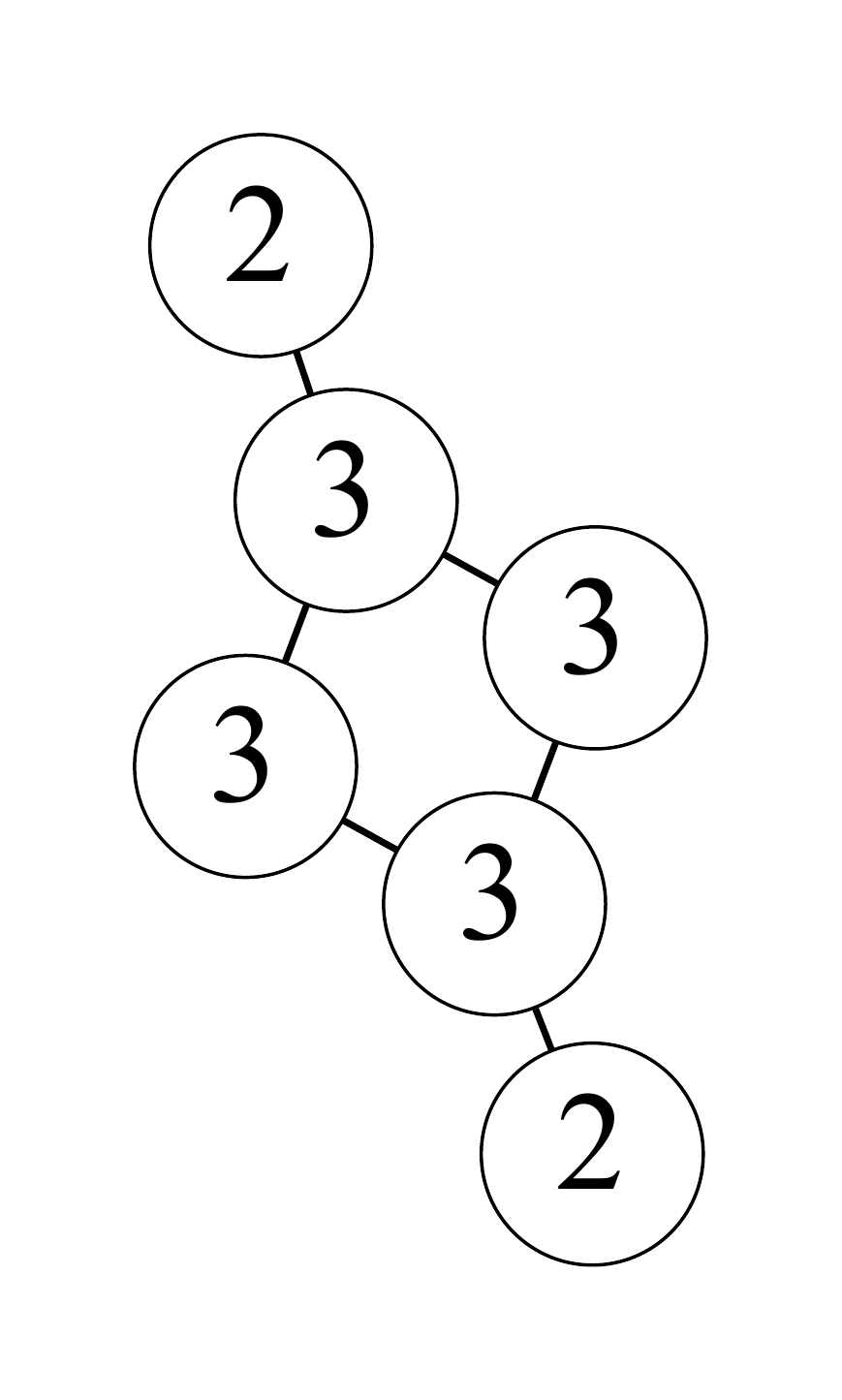}
     	\includegraphics[scale=0.25]{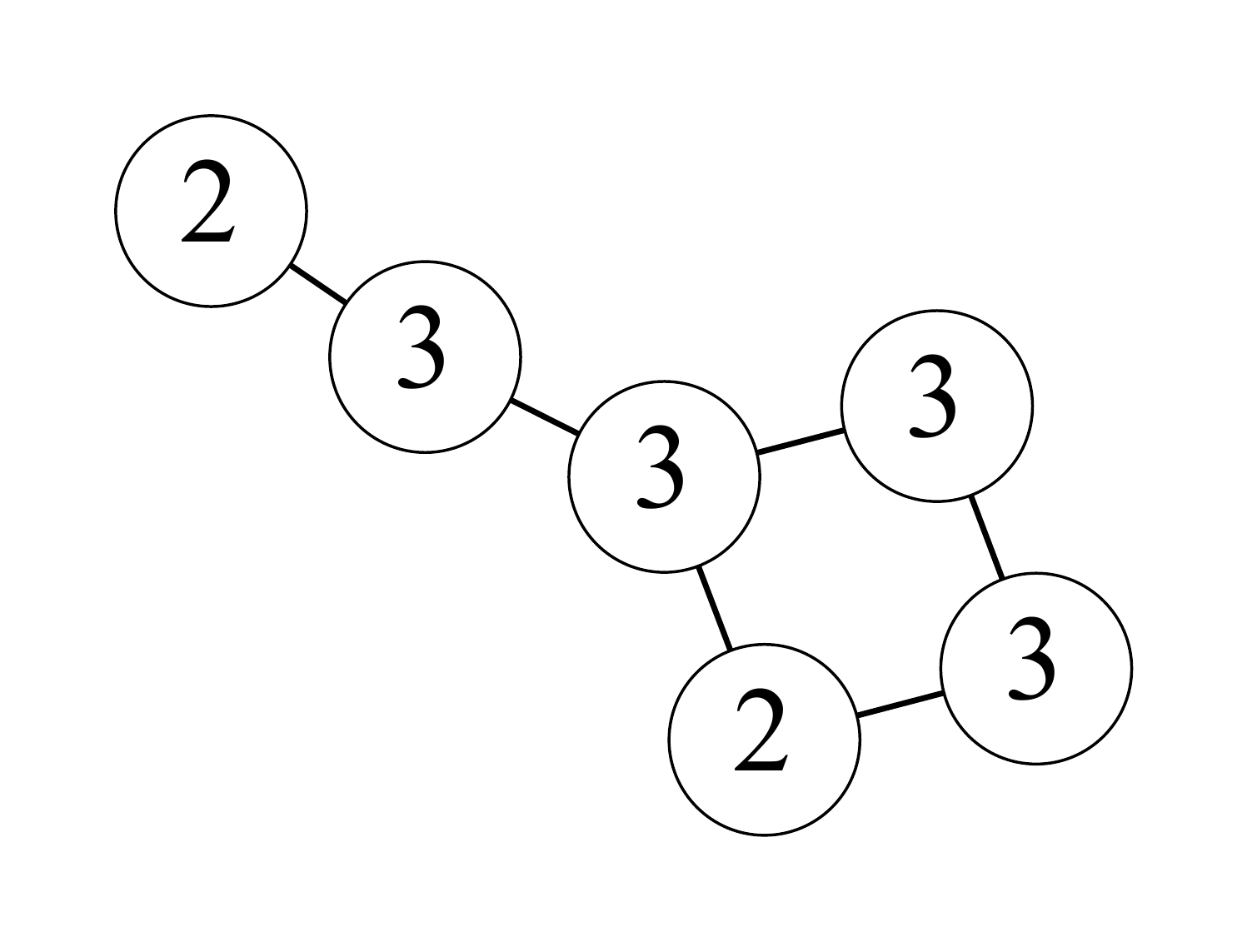}
     	\includegraphics[scale=0.25]{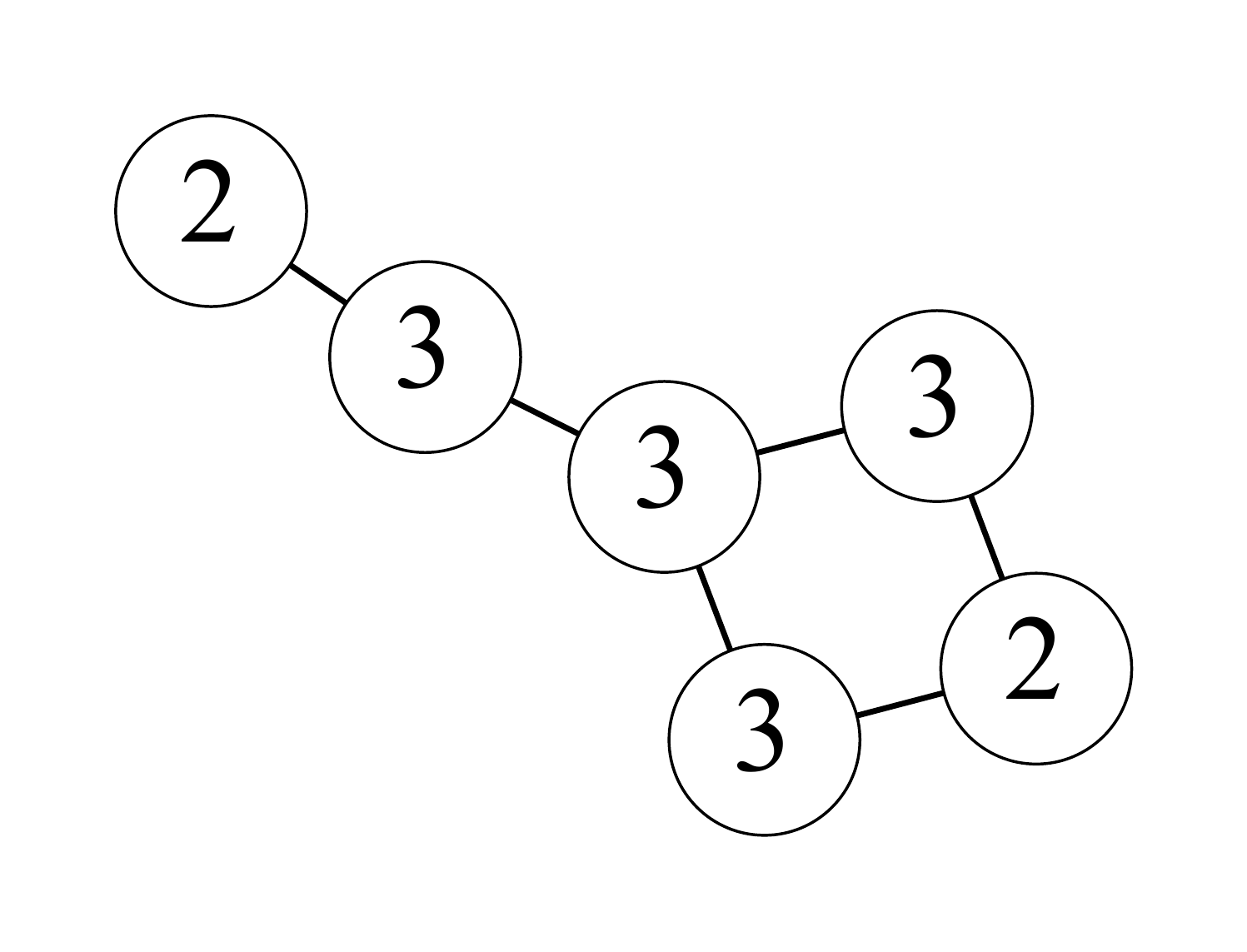}
     	\includegraphics[scale=0.25]{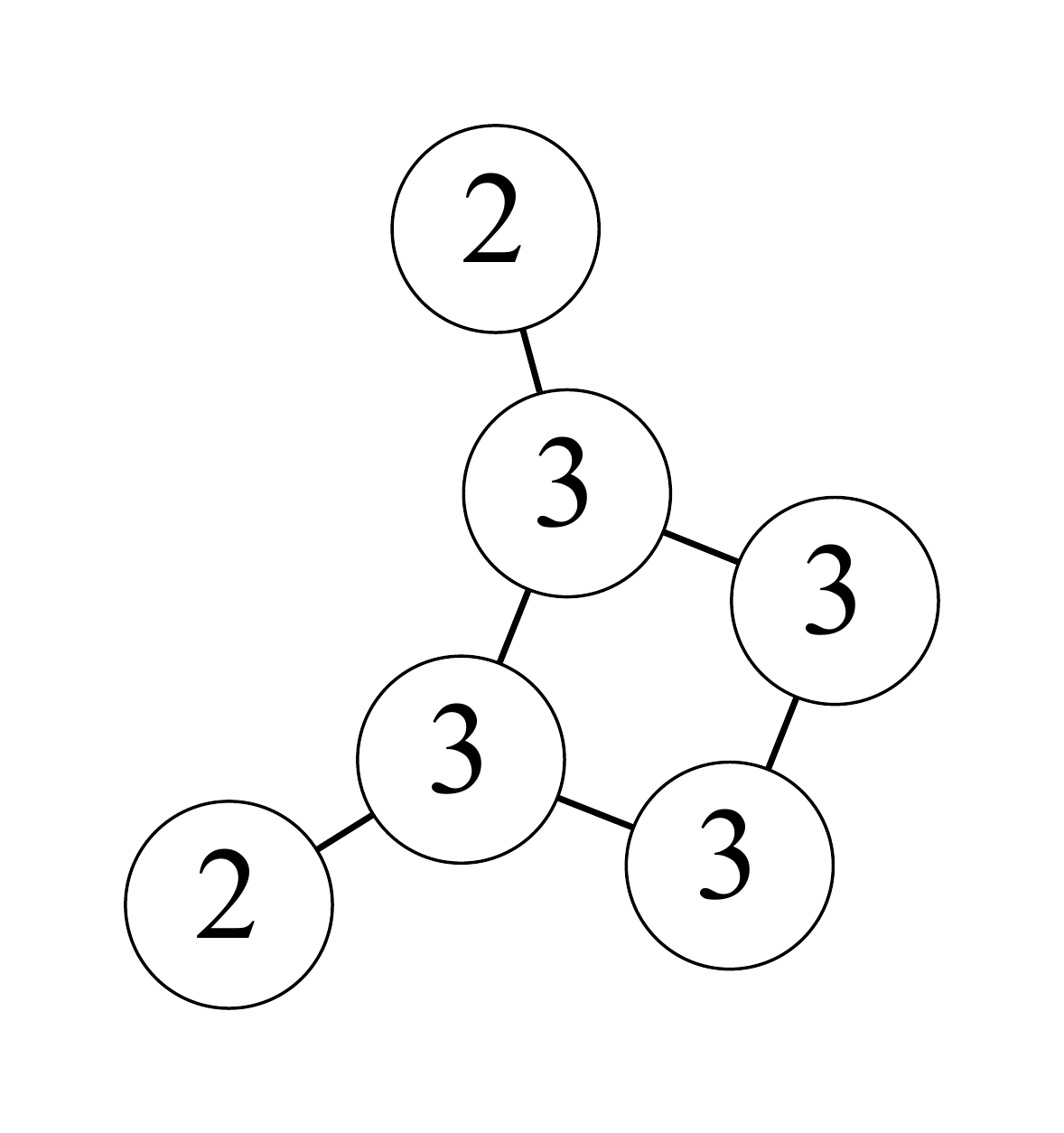}
     	\includegraphics[scale=0.25]{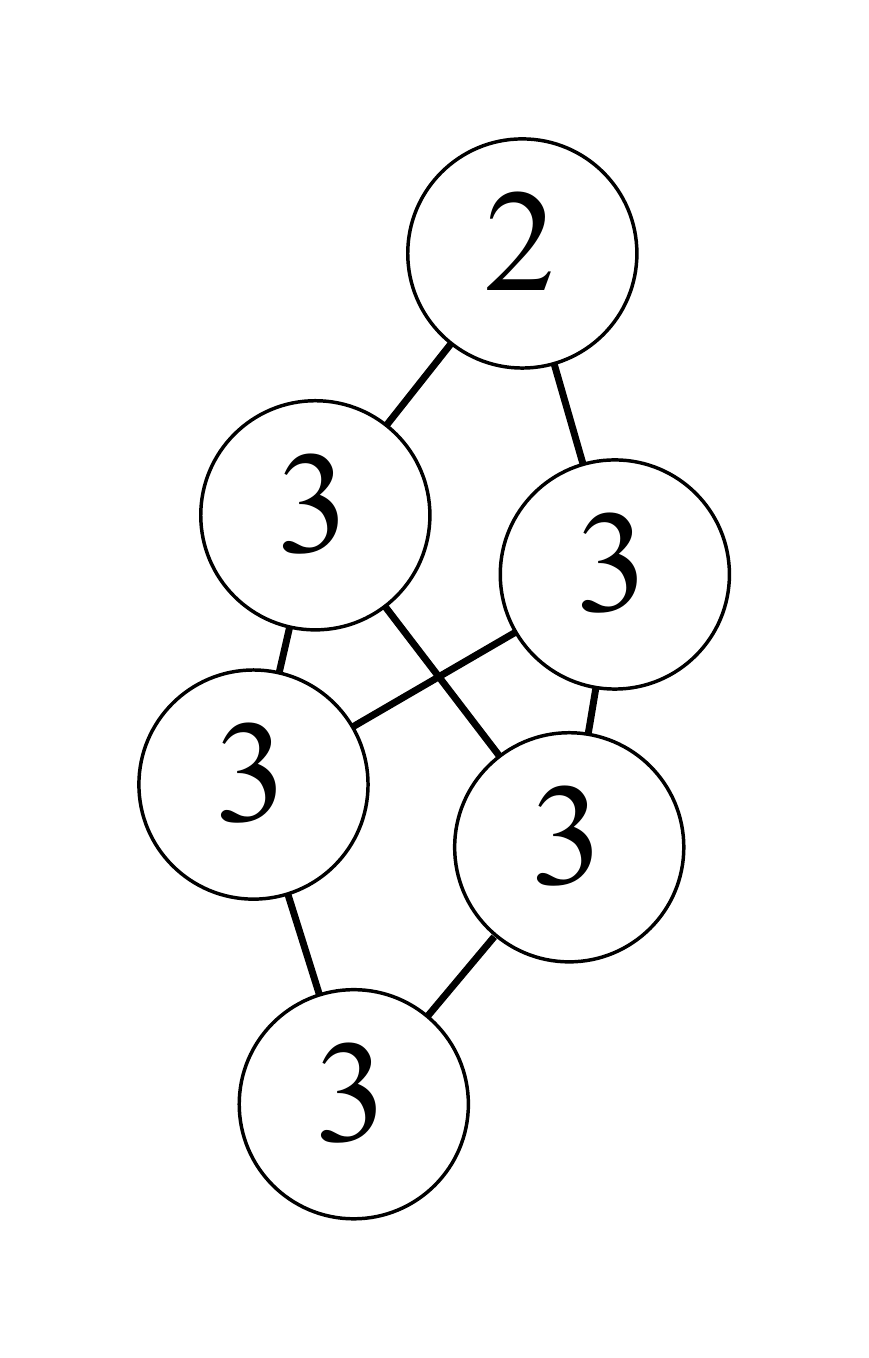}
     	\includegraphics[scale=0.25]{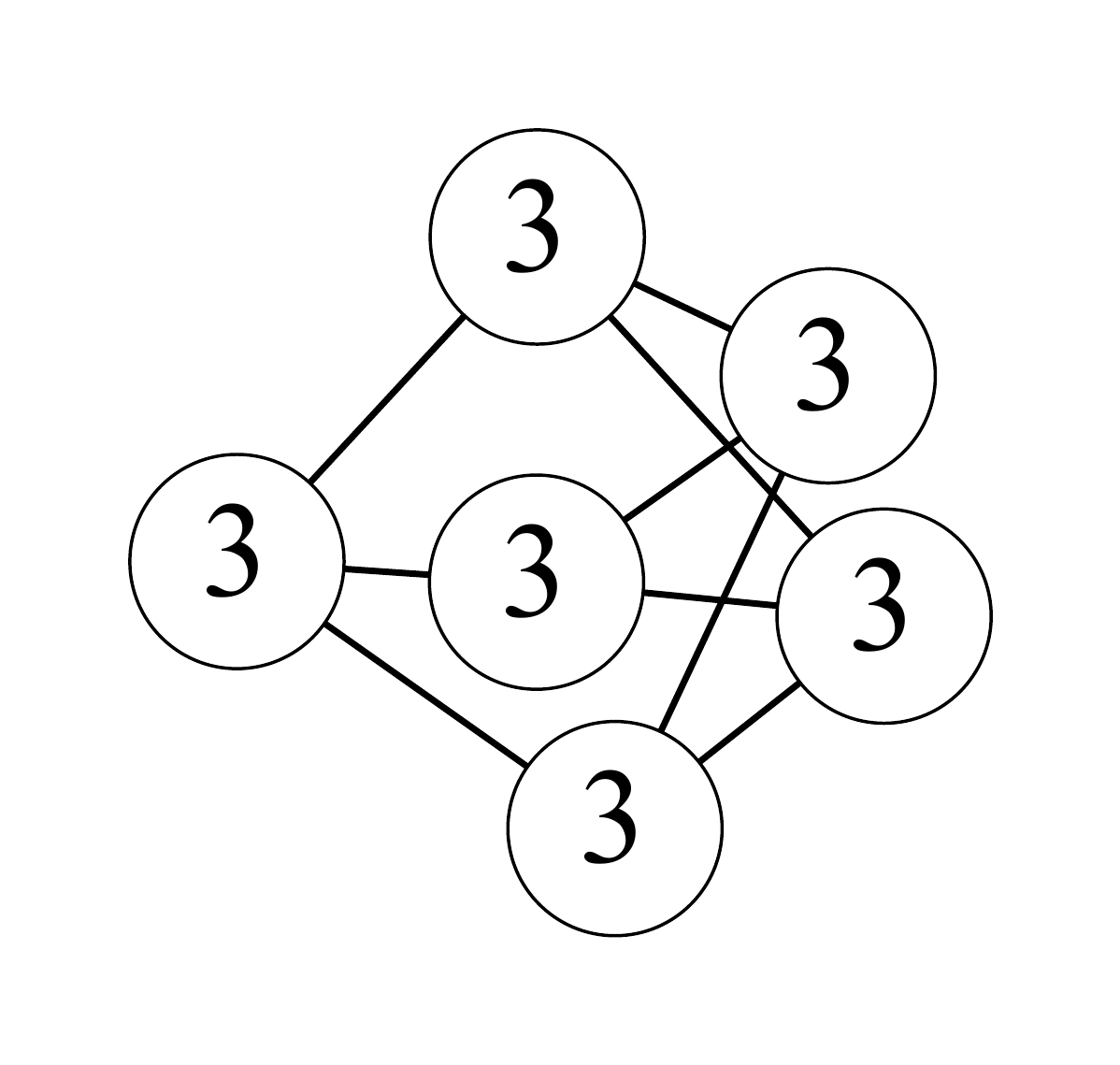}
     	\includegraphics[scale=0.25]{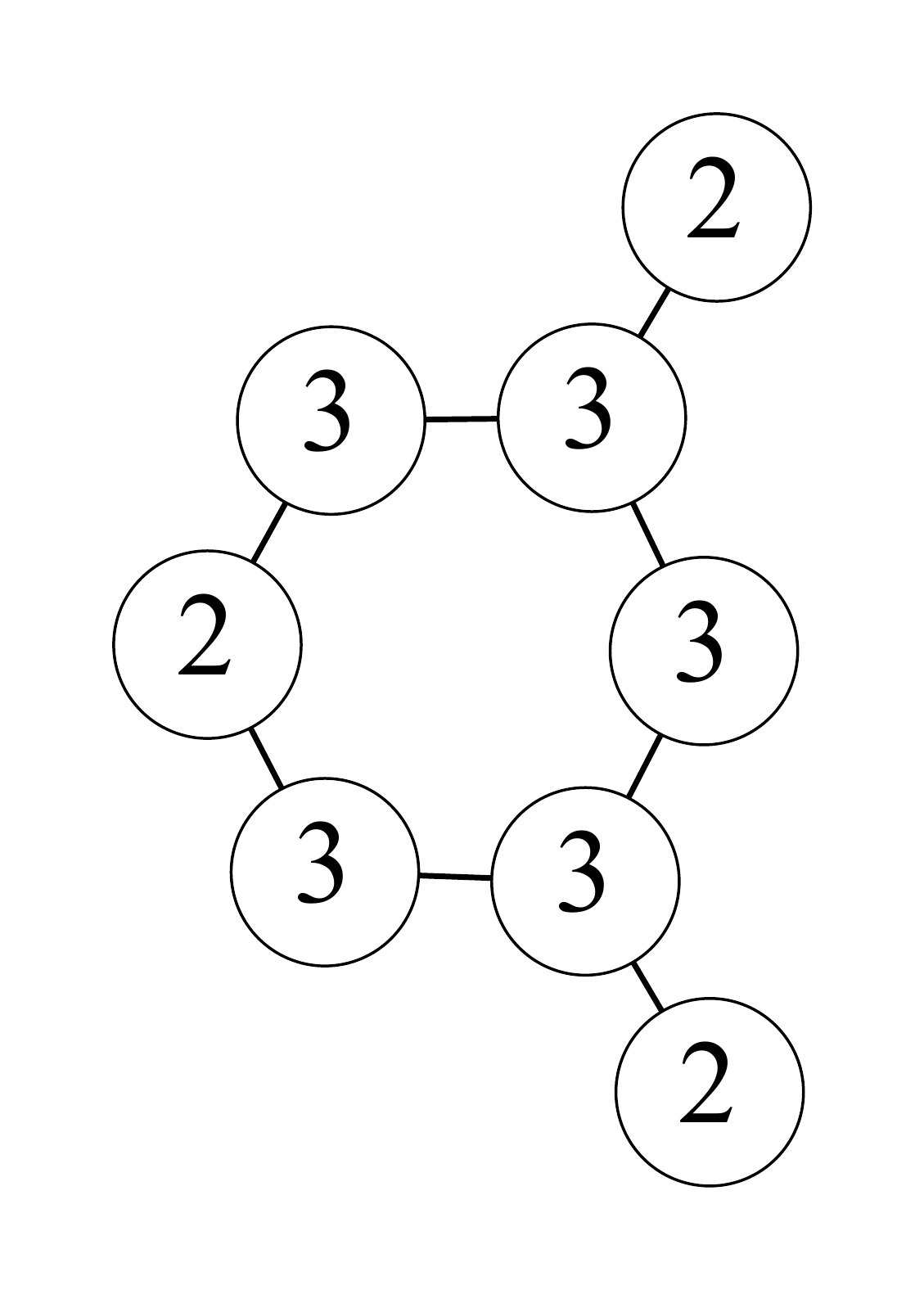}
		\caption{Some small $3$-fixable graphs. (The label for each vertex specifies
its degree in G.)\label{fig:small3}}
		\end{figure}

The penultimate graph in Figure \ref{fig:small3} is an example of the more
general fact that a $k$-regular graph with $f(v) = k$ for all $v$ is $k$-fixable
precisely when it is $k$-edge-colorable.  That the third graph in Figure
\ref{fig:small3} is reducible follows from Vizing's Adjacency Lemma.
\subsection{A necessary condition}
Since the edges incident to a vertex $v$ must all get different colors, 
%
if $G$ is $(L, P)$-fixable, then $|L(v)| \ge d_G(v)$ for all $v \in V(G)$.
%
By considering the maximum size of matchings in each color, we get a more
interesting necessary condition.
For each $C \subseteq \pot(L)$ and $H \subseteq G$, let $H_{L, C}$ be the
subgraph of $H$ induced by the vertices $v$ with $L(v) \cap C \ne \emptyset$. 
When $L$ is clear from context, we write $H_C$ for $H_{L,C}$. If $C =
\set{\alpha}$, we write $H_\alpha$ for $H_C$.  For $H \subseteq G$, let

\[\psi_L(H) = \sum_{\alpha \in \pot(L)} \floor{\frac{\card{H_{L, \alpha}}}{2}}.\]
Each term in the sum gives an upper bound on the size of a matching in color
$\alpha$. So $\psi_L(H)$ is an upper bound on the number of edges in a
partial $L$-edge-coloring of $H$.  The pair $(H, L)$ is \emph{abundant} if
$\psi_L(H) \ge \size{H}$ and $(G,L)$ is \emph{superabundant} if for every
$H \subseteq G$, the pair $(H, L)$ is abundant.  

\begin{lem}
\label{SuperabundanceIsNecessary} 
If $G$ is $(L, P)$-fixable, then $(G, L)$ is superabundant.
\end{lem}
\begin{proof}
Suppose instead that $G$ is $(L, P)$-fixable and there is $H \subseteq
G$ such that $(H, L)$ is not abundant. We show that for all distinct $a,b \in
P$ there is a partition $X_1, \ldots, X_t$ of $S_{a,b}$ into sets of size at
most two, such that for all $J \subseteq \irange{t}$, the pair $(H,L')$ is not
abundant, where $L'$ is formed from $L$ by swapping $a$ and $b$ in $L(v)$ for
every $v \in \bigcup_{i \in J} X_i$.  Since $G$ can only be edge-colored from
a superabundant list assignment, this contradicts that $G$ is $(L,P)$-fixable.

Pick distinct colors $a,b \in P$.  Let $S = S_{L,a,b} \cap V(H)$, let $S_a$ be the
$v \in S$ with $a \in L(v)$, and let $S_b = S\setminus S_a$.  
In the sum for $\psi_L(H)$, swapping $a$ and $b$ only effects the terms
$\floor{\frac{\card{S_a}}{2}}$ and $\floor{\frac{\card{S_b}}{2}}$.
So, if $\psi_L(H)$ is increased
by the swapping, it must be that both $|S_a|$ and $|S_b|$ are odd, and after
swapping they are both even.  Say $S_a = \set{a_1, \ldots,a_p}$ and $S_b =
\set{b_1, \ldots,b_q}$.  By symmetry, we assume $p \le q$.  For each $i \in
\irange{p}$, let $X_i = \set{a_i, b_i}$.  Since $p$ and $q$ are both odd, $q-p$
is even, so we get a partition by, for each $j \in \irange{\frac{q-p}{2}}$,
letting $X_{p + j} = \set{b_{p + 2j - 1}, b_{p + 2j}}$.  For any $i \in
\irange{p}$, swapping $a$ and $b$ in $L(v)$ for every $v \in X_i$ maintains
$|S_a|$ and $|S_b|$.  For any $j \in \irange{\frac{q-p}{2}}$, swapping $a$ and
$b$ in $L(v)$ for every $v \in X_{p+j}$ maintains the parity of $|S_a|$ and
$|S_b|$.  So no choice of $J$ can increase $\psi_L(H)$. Thus, $(H,L')$ is
never abundant.
\end{proof}

In particular, we conclude the following.

\begin{cor}
If $G$ is $(f,k)$-fixable, then $(G,L)$ is superabundant for every $L$ with
$L(v) \subseteq \irange{k}$ and $|L(v)| \ge k + d_{G}(v) - f(v)$ for all $v \in V(G)$.
\end{cor}

Intuitively, superabundance requires the potential for a large enough matching
in each color. If instead we require the existence of a large enough matching
in each color, then we get a stronger condition that has been studied before.
For a multigraph $H$, let $\nu(H)$ be the number of edges in a maximum matching
of $H$.  For a list assignment $L$ on $H$, let 
$$\eta_L(H) = \sum_{\alpha \in \pot(L)} \nu(H_\alpha).$$  
Note that always $\psi_L(H) \ge \eta_L(H)$.

The following generalization of Hall's theorem was proved by Marcotte and
Seymour \cite{marcotte1990extending} and independently by Cropper,
Gy{\'a}rf{\'a}s, and Lehel \cite{cropper2003edge}.  By a \emph{multitree} we
mean a tree that possibly has edges of multiplicity greater than one.

\begin{lem}[Marcotte and Seymour]\label{MultiTreeHall}
Let $T$ be a multitree and $L$ a list assignment on $V(T)$.  If $\eta_L(H) \ge
\size{H}$ for all $H \subseteq T$, then $T$ has an $L$-edge-coloring.
\end{lem}

In \cite{HallGame}, the second author proved that superabundance itself is also a
sufficient condition for fixability, when we restrict our graphs to be
multistars.  This result immediately implies the \emph{fan equation}, which is
an extension of Vizing's Adjacency Lemma to multigraphs and a standard tool in
proving reducibility for edge-coloring (see \cite[p.~19ff]{stiebitz2012graph}).
 The proof for multistars uses Hall's
theorem to reduce to a smaller star and one might hope that we could do the
same for arbitrary trees, with Lemma \ref{MultiTreeHall} in place of Hall's
theorem (thus giving a short proof that Tashkinov trees are elementary), but we
haven't yet made this work.

\subsection{Fixability of stars}
When $G$ is a star, superabundance implies fixability (provided that $|L(v)|\ge
d_G(v)$ for each vertex $v$), and this result generalizes Vizing fans
\cite{Vizing76}.  In \cite{HallGame}, the second author proved a
common generalization of 
this and of Hall's theorem; below we reproduce the proof for the special case
of edge-coloring.  In the next section we define ``Kierstead-Tashkinov-Vizing
assignments'' and show that they are always superabundant.  

\begin{thm}\label{FixabilityOfStars}
If $G$ is a multistar, then $G$ is $L$-fixable if and only if $(G, L)$ is
superabundant and $|L(v)| \ge d_G(v)$ for all $v \in V(G)$.
\end{thm}
\begin{proof}
Our strategy is simply to increase $\eta_L(G)$ if we can; if we cannot, then
Hall's theorem allows us to reduce to a smaller graph.  We can view
this strategy as the following double induction. Suppose the theorem is
false and choose a counterexample $(G, L)$ minimizing $\size{G}$ and, subject to
that, maximizing $\eta_L(G)$.  
	
Let $z$ be the center of the multistar $G$. Create a bipartite graph $B$ with
parts $C$ and $D$, where $C$
is the set of colors $\alpha$ that can be used on at least one edge, and $D$ is
the set of edges $e$ with at least one color available on $e$, and a color
$\alpha$ is adjacent to an edge $e$ if $\alpha$ can be used on $e$.
Note that $|C|=\eta_L(G)$.
	
First, suppose $|C| < \size{G}$.  Since $|L(z)| \ge d_G(z) = \size{G}$, 
some color $\tau \in L(z)$ cannot be used on any edge.  Suppose some
color $\beta \in C$ can be used on at least three edges.  
Let $zw$ be some edge that can use $\beta$.
Since $G$ is not $L$-fixable, there is $X \subseteq S_{L, \tau, \beta}$ with
$w \in X$ and $|X| \le 2$ such that $G$ is not $L'$-fixable, where $L'$ is formed
from $L$ by swapping $\tau$ and $\beta$ in $L(v)$ for every $v \in X$.  Since
$\beta$ can be used on at least three edges for $L$, it can be used on at
least one edge for $L'$.  Further, $\tau$ can also be used on at least one edge
for $L'$.  Thus $\eta_{L'}(G) > \eta_L(G)$.  Since $(G,L')$ is still
superabundant, this violates maximality of $\eta_L(G)$.  Hence, 
each color $\beta \in C$ can be used on at most two edges.  So, each color in $C$
contributes at most one to $\psi_L(G)$.  

Since $|C| < \size{G} \le \psi_L(G)$,
some color $\gamma$ contributes at least 1 to $\psi_L(G)$, but is not in $C$.
More precisely, some $\gamma \not \in C$ satisfies $|G_\gamma - z| \ge 2$.  Since $G$
is not $L$-fixable, there is $X \subseteq S_{L, \tau, \gamma}$ with $z \in X$ and 
$|X| \le 2$ such that $G$ is not $L'$-fixable where $L'$ is formed from $L$
by swapping $\tau$ and $\gamma$ in $L(v)$ for every $v \in X$.  Since
$\nu(G_{L, \tau}) = 0$ and $\nu(G_{L, \gamma}) = 0$ and $\nu(G_{L', \gamma}) =
1$, we have $\eta_{L'}(G) > \eta_L(G)$. Since $(G,L')$ is still superabundant,
this violates maximality of $\eta_L(G)$.  

Hence, we must have $|C| \ge \size{G}$.  In particular, $\card{N_B(C)} \le |C|$
so we may choose a set of colors $C' \subseteq C$ such that $C'$ is a minimal
nonempty set satisfying $\card{N_B(C')} \le \card{C'}$. 
If $|C'|\ge |N_B(C')|+1$, then, for any $\rho\in C'$, we have
$\card{C'-\rho}=\card{C'}-1\ge\card{N_B(C')}\ge\card{N_B(C'-\rho)}$, which
contradicts the minimality of $C'$.  Thus, $\card{C'}=\card{N_B(C')}$.
Furthermore, by minimality of $C'$, every nonempty $C''\subsetneq C'$ satisfies
$\card{N_B(C'')}>\card{C''}$, so Hall's Theorem yields a perfect matching $M$
between $C'$ and $N_B(C')$.

For each color/edge pair
$\set{\alpha, zw} \in M$, use color $\alpha$ on edge $zw$.  Form $G'$ from $G$
by removing all the colored edges and then discarding any isolated vertices.
Note that $z$ lost exactly $\card{C'}$ colors from its list and also
$d_{G'}(z)=d_G(z)-\card{C'}$, so $\card{L'(z)}=\card{L(z)}-\card{C'}\ge
d_G(z)-\card{C'}=d_{G'}(z)$.  Each other vertex $w\in V(G')$ satisifes
$d_{G'}(w)=d_G(w)$ and $\card{L'(w)}=\card{L(w)}$, so $\card{L'(w)}\ge
d_{G'}(w)$.
Since $G$ is not $L$-fixable and $C'$ and $\pot(L')$ are disjoint it must be
that $G'$ is not $L'$-fixable.  
For each $H \subseteq G'$, we have $\psi_{L'}(H) = \psi_{L}(H)$.
For each color $\alpha\in C$, if $\alpha\in C'$, then
$\floor{\card{H_{L,\alpha}}/2}=0$, since $E(H)\cap N_B(C')=\emptyset$.
Similarly, if $\alpha\notin C'$, then each $v\in V(G')$ satisfies $\alpha\in
L'(v)$ if and only if $\alpha\in L(v)$.
Thus, $H$ is abundant for $L'$ precisely because $H$ is abundant for $L$.
But $\size{G'} < \size{G}$, so by minimality of $\size{G}$, $G'$ is
$L'$-fixable, a contradiction.
\end{proof}

As shown in \cite{HallGame}, 
a direct consequence of Theorem \ref{FixabilityOfStars} is the {fan equation}. 
This, in turn, implies most classical edge-coloring results including Vizing's
Adjacency Lemma.

\subsection{Kierstead-Tashkinov-Vizing assignments}
Many edge-coloring results have been proved using a specific kind of
superabundant pair $(G, L)$ where superabundance can be proved via a special
ordering. That is, the orderings given by the definition of Vizing fans,
Kierstead paths, and Tashkinov trees (these structures are all standard tools in
edge-coloring; defenitions and more background are available
in~\cite{stiebitz2012graph}).
In this section, we show how superabundance follows easily from these orderings.
For each vertex $v$, we write $E(v)$ for the set of edges incident to $v$.

A list assignment $L$ on $G$ is a \emph{Kierstead-Tashkinov-Vizing} assignment
(henceforth \emph{KTV-assignment}) if for some edge $xy \in E(G)$, there is a
total ordering `$<$' of $V(G)$ such that

\begin{enumerate}
\item there is an edge-coloring $\pi$ of $G-xy$ such that $\pi(uv) \in L(u)
\cap L(v)$ for each edge $uv \in E(G - xy)$; 
\item $x < z$ for all $z \in V(G - x)$; 
\item $G\brackets{w \mid w \le z}$ is connected for all $z \in V(G)$; 
\item for each edge $wz \in E(G - xy)$, there is a vertex $u < \max\set{w, z}$ such that
$\pi(wz) \in L(u) - \setbs{\pi(e)}{e \in E(u)}$;
\item there are distinct vertices $s, t \in V(G)$ with $L(s) \cap L(t) -
\setbs{\pi(e)}{e \in E(s) \cup E(t)} \ne \emptyset$.
\end{enumerate}

\begin{lem}\label{KTVImpliesSuperabundant}
If $L$ is a KTV-assignment on $G$, then $(G, L)$ is superabundant.
\end{lem}
\begin{proof}
Let $L$ be a KTV-assignment on $G$, and let $H \subseteq G$.  We will show that
$(H,L)$ is abundant.  
Clearly it suffices to consider the case when $H$ is an induced subgraph, so we
assume this.
Property (1) gives that $G-xy$ has an edge-coloring
$\pi$, so $\psi_L(H)\ge \size{H}-1$; also $\psi_L(H)\ge \size{H}$ if
$\{x,y\}\not\subseteq V(H)$.  Furthermore $\psi_L(H)\ge \size{H}$ if $s$ and
$t$ from property (5) are both in $V(H)$, since then $\psi_L(H)$ gains 1 over
the naive lower bound, due to the color in $L(s)\cap L(t)$.  So $V(G)-
V(H)\ne \emptyset$.

Now choose a vertex $z \in V(G) - V(H)$ that is smallest under $<$.  
Let $H' = G\brackets{w \mid w \le z}$.  By the minimality of $z$, we have $H' -
z \subseteq H$. By property (2), $\card{H'} \ge 2$.  By property (3), $H'$ is
connected and thus there is $w \in V(H' - z)$ adjacent to $z$. So, we have $w <
z$ and $wz\in E(G)-E(H)$.  
Property (4) implies that there exists a vertex $u$ with $u <
\max\set{w, z} = z$ and $\pi(wz) \in L(u)-\{\pi(e)|e\in E(u)\}$.  Since $u \in
V(H' - z) \subseteq V(H)$, we again gain 1 over the naive lower bound on
$\psi_L(H)$, due to the color in $L(u)\cap L(w)$.  So $\psi_L(H)\ge \size{H}$.
\end{proof}

\subsection{The gap between fixability and reducibility}
By abstracting away the containing graph, we may have lost some power in proving
reducibility results. Surely we have when we only care about a certain class of
graphs. For example, with planar graphs, not all Kempe path pairings are
possible (if we add an edge for each pair, the resulting graph must be
planar).  But, possibly there are graphs that are reducible for all containing
graphs but are not fixable.  We could strengthen ``fixable'' in various ways,
but we have not found the need to do so.  One particular strengthening 
deserves mention, since it makes fixability more induction friendly.  

\begin{defn}
$G$ is \emph{$(L, P)$-subfixable} if either
\begin{enumerate}
\item[(1)] $G$ is $(L, P)$-fixable; or
\item[(2)] there is $xy \in E(G)$ and $\tau \in L(x) \cap L(y)$ such that
$G-xy$ is $L'$-subfixable, where $L'$ is formed from $L$ by removing $\tau$ from
$L(x)$ and $L(y)$.
\end{enumerate}
\end{defn}

Superabundance is a necessary condition for subfixability because coloring an
edge cannot make a non-abundant subgraph abundant.  The conjectures in the rest
of this paper may be easier to prove with subfixable in place of fixable.  That would
really be just as good since it would give the exact same results for edge coloring.

\section{Applications of small k-fixable graphs}
In this section, we use $k$-fixable graphs to prove a few conjectures about
3-critical and 4-critical graphs.
A \emph{$k$-vertex} is a vertex of degree $k$, and a \emph{$k$-neighbor} of a
vertex $v$ is a $k$-vertex adjacent to $v$.

\subsection{The conjecture of Hilton and Zhao for $\Delta=4$}
For a graph $G$, let $G_\Delta$ be the subgraph of $G$ induced by vertices of
degree $\Delta(G)$.  Vizing's Adjacency Lemma implies that $\delta(G_\Delta) \ge
2$ in a critical graph $G$.  A natural question is whether or not this is 
best possible.  For example, can we have $\Delta(G_\Delta) = 2$ in a
critical graph $G$?  In fact, Hilton and Zhao have conjectured exactly when
this can happen.  
Recall that a graph $G$ is \emph{class 1} if $\chi'(G)=\Delta$ and \emph{class
2} otherwise.
A graph $G$ is \emph{overfull} if $||G|| >
\floor{\frac{|G|}{2}}\Delta(G)$.  (The significance of overfull graphs is that
they must be class 2, simply because they have more edges than can be colored by
$\Delta(G)$ matchings, each of size $\floor{\frac{|G|}2}$.)
Let $P^*$ denote the Peterson graph with one
vertex deleted (see Figure \ref{fig:petey}).

\begin{conjecture}[Hilton and Zhao]
A connected graph $G$ with $\Delta(G_\Delta) \le 2$ is class 2 if and only if
$G$ is $P^*$ or $G$ is overfull.
\end{conjecture}

\tikzstyle{majorStyle}=[shape = circle, minimum size = 6pt, inner sep = 2.2pt, draw]
\tikzstyle{major}=[shape = circle, minimum size = 6pt, inner sep = 2.2pt, draw]
\tikzstyle{minorStyle}=[shape = rectangle, minimum size = 6pt, inner sep = 2.2pt, draw]
\tikzstyle{minor}=[shape = rectangle, minimum size = 6pt, inner sep = 2.2pt, draw]
\tikzstyle{labeledStyle}=[shape = rectangle, minimum size = 6pt, inner sep = 2.2pt, draw]
\tikzstyle{VertexStyle} = []
\tikzstyle{EdgeStyle} = []
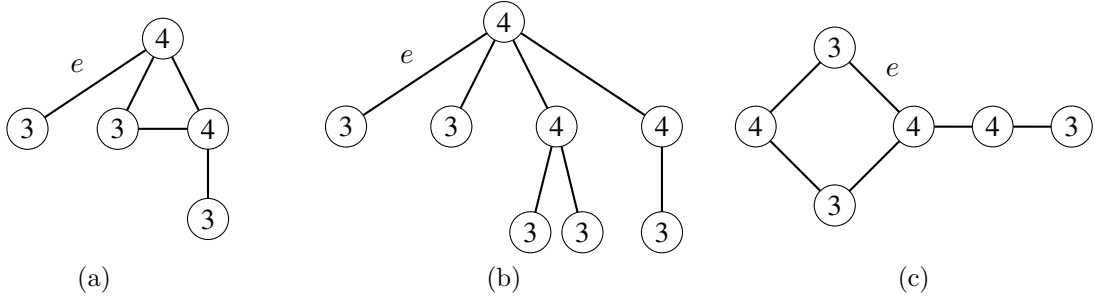
\begin{figure}[htb]
\renewcommand{\ttdefault}{ptm}
\begin{center}
\subfloat[]{\makebox[.33\textwidth]{
\begin{tikzpicture}[scale = 6]
\Vertex[style = major, x = 0.45, y = 0.95, L = \small {\texttt{4}}]{v0}
\Vertex[style = major, x = 0.15, y = 0.75, L = \small {\texttt{3}}]{v1}
\Vertex[style = major, x = 0.35, y = 0.75, L = \small {\texttt{3}}]{v2}
\Vertex[style = major, x = 0.55, y = 0.75, L = \small {\texttt{4}}]{v3}
\Vertex[style = major, x = 0.55, y = 0.55, L = \small {\texttt{3}}]{v4}
\Edge[label= \small {$e$}, , labelstyle={auto=right, fill=none}](v0)(v1)
\Edge[](v2)(v0)
\Edge[](v2)(v3)
\Edge[](v3)(v0)
\Edge[](v4)(v3)
\draw[white] (0,.475)--(0,0.49);
\end{tikzpicture}
}}
\subfloat[]{\makebox[.33\textwidth]{
\begin{tikzpicture}[scale = 7]
\Vertex[style = major, x = 0.5, y = 0.849, L = \small {\texttt{4}}]{v0}
\Vertex[style = major, x = 0.200, y = 0.650, L = \small {\texttt{3}}]{v1}
\Vertex[style = major, x = 0.400, y = 0.650, L = \small {\texttt{3}}]{v2}
\Vertex[style = major, x = 0.600, y = 0.650, L = \small {\texttt{4}}]{v3}
\Vertex[style = major, x = 0.800, y = 0.650, L = \small {\texttt{4}}]{v4}
\Vertex[style = major, x = 0.550, y = 0.449, L = \small {\texttt{3}}]{v6}
\Vertex[style = major, x = 0.649, y = 0.449, L = \small {\texttt{3}}]{v5}
\Vertex[style = major, x = 0.800, y = 0.449, L = \small {\texttt{3}}]{v7}
\Edge[label= \small {$e$}, , labelstyle={auto=right, fill=none}](v0)(v1)
\Edge[](v2)(v0)
\Edge[](v3)(v0)
\Edge[](v4)(v0)
\Edge[](v5)(v3)
\Edge[](v6)(v3)
\Edge[](v7)(v4)
\end{tikzpicture}
}}
%
\subfloat[]{\makebox[.33\textwidth]{
\begin{tikzpicture}[rotate=90,scale=7]
\Vertex[style = major, x = 0.35, y = 0.80, L = \small {\texttt{3}}]{v0}
\Vertex[style = major, x = 0.65, y = 0.80, L = \small {\texttt{3}}]{v1}
\Vertex[style = major, x = 0.50, y = 0.65, L = \small {\texttt{4}}]{v2}
\Vertex[style = major, x = 0.50, y = 0.95, L = \small {\texttt{4}}]{v3}
\Vertex[style = major, x = 0.50, y = 0.50, L = \small {\texttt{4}}]{v4}
\Vertex[style = major, x = 0.50, y = 0.35, L = \small {\texttt{3}}]{v5}
\Edge[](v2)(v0)
\Edge[label= \small {$e$}, , labelstyle={auto=right, fill=none}](v2)(v1)
\Edge[](v3)(v0)
\Edge[](v3)(v1)
\Edge[](v4)(v2)
\Edge[](v5)(v4)
\draw[white] (.26,.55)--(.27,.55);
\end{tikzpicture}
}}
\end{center}
\caption{Each configuration is reducible by deleting edge $e$.
\label{fig:hiltonzhao}}
\end{figure}

David and Gianfranco Cariolaro~\cite{cariolaro2003colouring} proved this
conjecture when $\Delta=3$.  Here we prove it when $\Delta=4$, but we
omit the very long computer-generated proofs of the reducibility of the
graphs in Figure~\ref{fig:hiltonzhao}.  
%
Since we do not include the reducibility proofs, we separate the proof into two
parts.  The first does not use the computer at all.
Let $\H_4$ be the class of connected graphs with maximum degree 4, minimum
degree 3, each vertex adjacent to at least two 4-vertices, and each 4-vertex
adjacent to exactly two 4-vertices.
\begin{lem}\label{HiltonZhaoLemma}
If $G$ is a graph in $\H_4$ and $G$ contains none of the three configurations in
Figure~\ref{fig:hiltonzhao} (not necessarily induced), then $G$ is $K_5-e$.
\end{lem}
\begin{proof}
Let $G$ be a graph in $\H_4$.  Note that every 4-vertex in $G$ has exactly two
3-neighbors and two 4-neighbors.  Let $u$ denote a 4-vertex and let
$v_1,\ldots,v_4$ denote its neighbors, where $d(v_1)=d(v_2)=3$ and $d(v_3)=d(v_4)=4$.
When vertices $x$ and $y$ are adjacent, we write $x\adj y$.  We assume that $G$
contains none of the configurations in Figure~\ref{fig:hiltonzhao} and show that
$G$ must be $K_5-e$.  
	
First suppose that $u$ has a 3-neighbor and a 4-neighbor that are adjacent.  By
symmetry, assume that $v_2\adj v_3$.  Since
Figure~\ref{fig:hiltonzhao}(a) is forbidden, we have $v_3\adj v_1$. 
Now consider $v_4$.  If $v_4$ has a 3-neighbor distinct from $v_1$ and $v_2$,
then we have a copy of Figure~\ref{fig:hiltonzhao}(c).  Hence $v_4\adj v_1$ and
$v_4\adj v_2$.  If $v_3\adj v_4$, then $G$ is $K_5-e$.  Suppose not, and let
$x$ be a 4-neighbor of $v_4$.  Since $G$ has no copy of
Figure~\ref{fig:hiltonzhao}(c), $x$ must be adjacent to $v_1$ and $v_2$.  This
is a contradiction, since $v_1$ and $v_2$ are 3-vertices, but now each has at
least four neighbors.  Hence, we conclude that each of $v_1$ and $v_2$ is
non-adjacent to each of $v_3$ and $v_4$.
	
Now consider the 3-neighbors of $v_3$ and $v_4$.  If $v_3$ and $v_4$ have zero 
or one 3-neighbors
in common, then we have a copy of Figure~\ref{fig:hiltonzhao}(b).  
Otherwise they have two 3-neighbors in common,
so we have a copy of Figure~\ref{fig:hiltonzhao}(c).  
\end{proof}

Since $K_5 - e$ is overfull, the next theorem implies Hilton and Zhao's conjecture
for $\Delta=4$.  

\begin{thm}
A connected graph $G$ with $\Delta(G) = 4$ and $\Delta(G_\Delta) \le 2$ is
class 2 if and only if $G$ is $K_5-e$.
\end{thm}
%
%
%

\begin{proof} 
Let $G$ be as stated in the theorem.
If $G$ is class $2$, then $G$ has a $4$-critical subgraph $H$.   Since
$H$ is $4$-critical, it is connected, 
and every vertex has at least two neighbors of degree $4$, by VAL.  
Further, since $\Delta(H_\Delta) \le \Delta(G_\Delta) \le 2$, VAL implies
that $H$ has minimum degree $3$. 
Thus, $H \in \H_4$.  By Lemma \ref{HiltonZhaoLemma}, 
either $H$ is $K_5-e$ or $H$ contains one of
the configurations in Figure~\ref{fig:hiltonzhao}.  By computer, each of these
configurations is reducible and hence cannot be a subgraph of the $4$-critical
graph $H$.  Thus $H$ is $K_5-e$.  Let $x_1,x_2$ be the degree $3$ vertices in
$H$.  Each $x_i$ has three degree $4$ neighbors in $H$ and hence $d_G(x_i) \le
3$ since $\Delta(G_\Delta) \le 2$.  That is, $x_i$ has no neighbors outside
$H$.  Since $G$ is connected, we must have $G = H = K_5 - e$. 
\end{proof}

\subsection{Impoved lower bounds on the average degree of 3-critical graphs and
4-critical graphs}
Let $P^*$ denote the Petersen graph with a vertex deleted (see Figure \ref{fig:petey}).
Jakobsen~\cite{Jakobsen73,Jakobsen74} noted that $P^*$ is 3-critical and has
average degree $2.\overline{6}$.  He showed that every 3-critical graph has average
degree at least $2.\overline{6}$, and asked whether equality holds only for $P^*$.
In~\cite{3criticalCR}, we answered his question affirmatively.  More precisely,
we showed that every 3-critical graph other than $P^*$ has average degree at
least $2+\frac{26}{37}=2.\overline{702}$.  The proof crucially depends on the
fact that the three leftmost configurations in Figure~\ref{tree1-pic} are
reducible for 3-edge-coloring.
As we noted in~\cite{3criticalCR}, by using the computer to prove reducibility
of additional configurations, we can slightly strengthen this result.  Specifically,
every 3-critical graph has average degree at least $2+\frac{22}{31} \approx
2.7097$ unless it is $P^*$ or one other exceptional graph, the Haj\'{o}s
join of two copies of $P^*$.  (For comparison, there exists an infinite family of
3-critical graphs with average degree less than $2.75$.)  This strengthening
relies primarily on the fact that the rightmost configuration in
Figure~\ref{fig:bigtree} is reducible, even if one or more pairs of
its 2-vertices are identified.  However, the simplest proof we have of this fact
is computer-generated and fills about 100 pages.

\begin{figure}[!htb]
\begin{center}
\begin{tikzpicture}[scale = 8, font=\sffamily]
\tikzstyle{VertexStyle} = []
\tikzstyle{EdgeStyle} = []
\tikzstyle{labeledStyle}=[shape = circle, minimum size = 6pt, inner sep = 2.2pt, draw]
\tikzstyle{unlabeledStyle}=[shape = circle, minimum size = 6pt, inner sep = 1.2pt, draw, fill]
\Vertex[style = unlabeledStyle, x = 0.45, y = 0.80, L = \tiny {}]{v0}
\Vertex[style = unlabeledStyle, x = 0.35, y = 0.75, L = \tiny {}]{v1}
\Vertex[style = unlabeledStyle, x = 0.30, y = 0.65, L = \tiny {}]{v2}
\Vertex[style = unlabeledStyle, x = 0.35, y = 0.55, L = \tiny {}]{v3}
\Vertex[style = unlabeledStyle, x = 0.45, y = 0.50, L = \tiny {}]{v4}
\Vertex[style = unlabeledStyle, x = 0.55, y = 0.55, L = \tiny {}]{v5}
\Vertex[style = unlabeledStyle, x = 0.60, y = 0.65, L = \tiny {}]{v6}
\Vertex[style = unlabeledStyle, x = 0.55, y = 0.75, L = \tiny {}]{v7}
\Vertex[style = unlabeledStyle, x = 0.45, y = 0.95, L = \tiny {}]{v8}
\Edge[label = \tiny {}, labelstyle={auto=right, fill=none}](v1)(v2)
\Edge[label = \tiny {}, labelstyle={auto=right, fill=none}](v1)(v0)
\Edge[label = \tiny {}, labelstyle={auto=right, fill=none}](v7)(v0)
\Edge[label = \tiny {}, labelstyle={auto=right, fill=none}](v7)(v6)
\Edge[label = \tiny {}, labelstyle={auto=right, fill=none}](v7)(v3)
\Edge[label = \tiny {}, labelstyle={auto=right, fill=none}](v5)(v6)
\Edge[label = \tiny {}, labelstyle={auto=right, fill=none}](v5)(v4)
\Edge[label = \tiny {}, labelstyle={auto=right, fill=none}](v5)(v1)
\Edge[label = \tiny {}, labelstyle={auto=right, fill=none}](v3)(v4)
\Edge[label = \tiny {}, labelstyle={auto=right, fill=none}](v3)(v2)
\path [line width=1pt]
(v2) edge [bend left=35] (v8)
(v8) edge [bend left=35] (v6);
\end{tikzpicture}
\end{center}
\caption{The Peterson graph with one vertex removed.}
\label{fig:petey}
\end{figure}
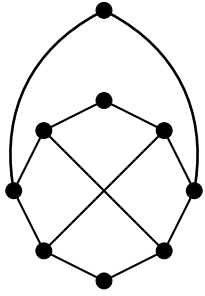
\bigskip

\begin{figure}[!htb]
\renewcommand{\ttdefault}{ptm}
\begin{center}
\begin{tikzpicture}[scale = 9, font=\sffamily]
\tikzstyle{VertexStyle} = []
\tikzstyle{EdgeStyle} = []
\tikzstyle{labeledStyle}=[shape = circle, minimum size = 6pt, inner sep = 2.2pt, draw]
\tikzstyle{unlabeledStyle}=[shape = circle, minimum size = 6pt, inner sep = 1.2pt, draw, fill]
\Vertex[style = labeledStyle, x = 0.45, y = 0.75, L = \small {\texttt{2}}]{v0}
\Vertex[style = labeledStyle, x = 0.45, y = 0.65, L = \small {\texttt{3}}]{v1}
\Vertex[style = labeledStyle, x = 0.55, y = 0.65, L = \small {\texttt{3}}]{v2}
\Vertex[style = labeledStyle, x = 0.55, y = 0.75, L = \small {\texttt{2}}]{v3}
\Vertex[style = labeledStyle, x = 0.65, y = 0.75, L = \small {\texttt{2}}]{v4}
\Vertex[style = labeledStyle, x = 0.65, y = 0.65, L = \small {\texttt{3}}]{v5}
\Vertex[style = labeledStyle, x = 0.75, y = 0.65, L = \small {\texttt{3}}]{v6}
\Vertex[style = labeledStyle, x = 0.75, y = 0.75, L = \small {\texttt{2}}]{v7}
\Edge[](v6)(v7)
\Edge[](v6)(v5)
\Edge[](v5)(v4)
\Edge[](v5)(v2)
\Edge[](v2)(v3)
\Edge[](v2)(v1)
\Edge[](v1)(v0)
\end{tikzpicture}
~~~
\begin{tikzpicture}[scale = 9, font=\sffamily]
\tikzstyle{VertexStyle} = []
\tikzstyle{EdgeStyle} = []
\tikzstyle{labeledStyle}=[shape = circle, minimum size = 6pt, inner sep = 2.2pt, draw]
\tikzstyle{unlabeledStyle}=[shape = circle, minimum size = 6pt, inner sep = 1.2pt, draw, fill]
\Vertex[style = labeledStyle, x = 0.60, y = 0.80, L = \small {\texttt{2}}]{v0}
\Vertex[style = labeledStyle, x = 0.50, y = 0.70, L = \small {\texttt{3}}]{v1}
\Vertex[style = labeledStyle, x = 0.60, y = 0.70, L = \small {\texttt{3}}]{v2}
\Vertex[style = labeledStyle, x = 0.70, y = 0.70, L = \small {\texttt{3}}]{v3}
\Vertex[style = labeledStyle, x = 0.60, y = 0.60, L = \small {\texttt{2}}]{v4}
\Edge[](v3)(v2)
\Edge[](v3)(v0)
\Edge[](v1)(v0)
\Edge[](v1)(v2)
\Edge[](v2)(v4)
\end{tikzpicture}
~~~
\begin{tikzpicture}[scale = 9, font=\sffamily]
\tikzstyle{VertexStyle} = []
\tikzstyle{EdgeStyle} = []
\tikzstyle{labeledStyle}=[shape = circle, minimum size = 6pt, inner sep = 2.2pt, draw]
\tikzstyle{unlabeledStyle}=[shape = circle, minimum size = 6pt, inner sep = 1.2pt, draw, fill]
\Vertex[style = labeledStyle, x = 0.60, y = 0.80, L = \small {\texttt{3}}]{v0}
\Vertex[style = labeledStyle, x = 0.50, y = 0.70, L = \small {\texttt{3}}]{v1}
\Vertex[style = labeledStyle, x = 0.60, y = 0.70, L = \small {\texttt{3}}]{v2}
\Vertex[style = labeledStyle, x = 0.70, y = 0.70, L = \small {\texttt{3}}]{v3}
\Vertex[style = labeledStyle, x = 0.70, y = 0.60, L = \small {\texttt{2}}]{v4}
\Vertex[style = labeledStyle, x = 0.60, y = 0.60, L = \small {\texttt{2}}]{v5}
\Vertex[style = labeledStyle, x = 0.50, y = 0.60, L = \small {\texttt{2}}]{v6}
\Edge[](v3)(v4)
\Edge[](v3)(v2)
\Edge[](v3)(v0)
\Edge[](v1)(v6)
\Edge[](v1)(v0)
\Edge[](v1)(v2)
\Edge[](v2)(v5)
\end{tikzpicture}
~~~
\begin{tikzpicture}[scale = 9, font=\sffamily]
\tikzstyle{VertexStyle} = []
\tikzstyle{EdgeStyle} = []
\tikzstyle{labeledStyle}=[shape = circle, minimum size = 6pt, inner sep = 2.2pt, draw]
\tikzstyle{unlabeledStyle}=[shape = circle, minimum size = 6pt, inner sep = 1.2pt, draw, fill]
\Vertex[style = labeledStyle, x = 0.70, y = 0.70, L = \small {\texttt{3}}]{v0}
\Vertex[style = labeledStyle, x = 0.80, y = 0.70, L = \small {\texttt{3}}]{v1}
\Vertex[style = labeledStyle, x = 0.90, y = 0.70, L = \small {\texttt{3}}]{v2}
\Vertex[style = labeledStyle, x = 0.80, y = 0.60, L = \small {\texttt{2}}]{v3}
\Vertex[style = labeledStyle, x = 0.70, y = 0.60, L = \small {\texttt{2}}]{v4}
\Vertex[style = labeledStyle, x = 0.90, y = 0.60, L = \small {\texttt{2}}]{v5}
\Vertex[style = labeledStyle, x = 0.30, y = 0.70, L = \small {\texttt{3}}]{v6}
\Vertex[style = labeledStyle, x = 0.40, y = 0.70, L = \small {\texttt{3}}]{v7}
\Vertex[style = labeledStyle, x = 0.50, y = 0.70, L = \small {\texttt{3}}]{v8}
\Vertex[style = labeledStyle, x = 0.30, y = 0.60, L = \small {\texttt{2}}]{v9}
\Vertex[style = labeledStyle, x = 0.50, y = 0.60, L = \small {\texttt{2}}]{v10}
\Vertex[style = labeledStyle, x = 0.40, y = 0.60, L = \small {\texttt{2}}]{v11}
\Vertex[style = labeledStyle, x = 0.60, y = 0.80, L = \small {\texttt{3}}]{v12}
\Edge[label = \tiny {}, labelstyle={auto=right, fill=none}](v1)(v0)
\Edge[label = \tiny {}, labelstyle={auto=right, fill=none}](v1)(v2)
\Edge[label = \tiny {}, labelstyle={auto=right, fill=none}](v3)(v1)
\Edge[label = \tiny {}, labelstyle={auto=right, fill=none}](v4)(v0)
\Edge[label = \tiny {}, labelstyle={auto=right, fill=none}](v5)(v2)
\Edge[label = \tiny {}, labelstyle={auto=right, fill=none}](v6)(v7)
\Edge[label = \tiny {}, labelstyle={auto=right, fill=none}](v8)(v7)
\Edge[label = \tiny {}, labelstyle={auto=right, fill=none}](v8)(v10)
\Edge[label = \tiny {}, labelstyle={auto=right, fill=none}](v9)(v6)
\Edge[label = \tiny {}, labelstyle={auto=right, fill=none}](v11)(v7)
\Edge[label = \tiny {}, labelstyle={auto=right, fill=none}](v12)(v0)
\Edge[label = \tiny {}, labelstyle={auto=right, fill=none}](v12)(v8)
\end{tikzpicture}
\caption{Four subgraphs forbidden from a 3-critical graph $G$.}
\label{tree1-pic}
\label{umbrella-pic}
\label{jellyfish-pic}
\label{fig:bigtree}
\end{center}
\end{figure}
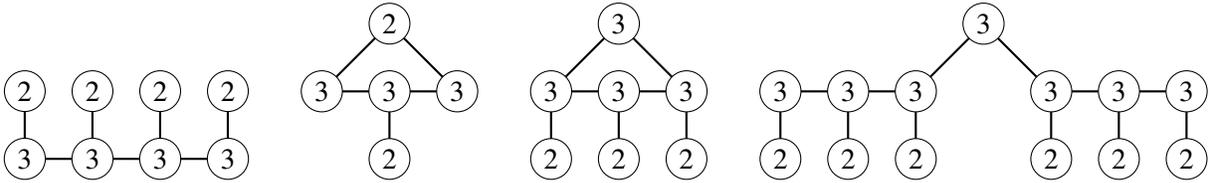

Woodall conjectured \cite{woodall2008average} that the average degree of every
4-critical graph is at least 3.6, which is best possible due to $K_5-e$.
We have proved this conjecture (modulo computer proofs of reducibility).
However, the proof requires 39 reducible configurations, so we defer it to an
appendix; even there, we omit the computer-aided reducibility proofs.
Here, we give a brief outline to illustrate the technique.

Our proof uses the discharging method.  Assume that $G$ is a 4-critical graph.
Each vertex begins with an initial charge, which is its degree.  We redistribute
the charge (without changing its sum), with the goal that every vertex finishes
with charge at least 3.6.  If this is true, then $G$ has the desired average
degree.  To redistribute charge, we successively apply the following 3 rules.

\begin{enumerate}
\item[(R1)] Each 2-vertex takes $.8$ from each 4-neighbor.
\item[(R2)] Each 3-vertex with three 4-neighbors takes $.2$ from each 4-neighbor.
Each 3-vertex with two 4-neighbors takes $.3$ from each 4-neighbor.
\item[(R3)] Each 4-vertex with charge in excess of $3.6$ after (R2) splits this
		excess evenly among its 4-neighbors with charge less than $3.6$.
\end{enumerate}
	
	
By Vizing's Adjacency Lemma (VAL), each neighbor of a 2-vertex $v$ is a
4-neighbor.  Thus, $v$ finishes with charge at least $2+2(.8)=3.6$.
Again by VAL, each 3-vertex $v$ has at least two 4-neighbors. So $v$ finishes
with charge at least $3+3(.2)$ or $3+2(.3)$, both of which are at least 3.6.

It is also easy to check that each 4-vertex $v$ finishes with charge at least 3.2;
by VAL, $v$ has at least two 4-neighbors, and if it has a 2-neighbor, then it has
three 4-neighbors.  So the remainder of the proof consists in showing that all
4-vertices that finish (R2) with charge less than 3.6 receive enough charge by
(R3).  The intuition is simple: if $v$ has few low degree neighbors and
neighbors of neighbors, then $v$ gets enough charge; otherwise, $v$ is contained
in some reducible configuration, which contradicts our choice of $G$ as
4-critical.

\section{Superabundance sufficiency and adjacency lemmas}
In the previous sections, we studied $k$-fixable graphs, which are reducible
configurations for graphs with fixed maximum degree.  Here we study a more
general notion that behaves similarly to Vizing Fans, Kierstead Paths, and
Tashkinov Trees.  Specifically, we consider graphs that are fixable for all
superabundant list assignments.  

\subsection{Superabundant fixability in general}
	
\begin{defn}
If $G$ is a graph and $\func{f}{V(G)}{\IN}$ with $f(v) \ge d_G(v)$ for all $v
\in V(G)$, then $G$ is $f$-fixable if $G$ is $(L, P)$-fixable for every $L$
with $|L(v)| \ge f(v)$ for all $v \in V(G)$ and every $L$-pot $P$ such that
$(G,L)$ is superabundant.  If a graph $G$ is $f$-fixable when $f(v)=d_G(v)$ for
each $v$, then $G$ is \emph{degree-fixable}.
\end{defn}
	
For example, Lemma \ref{FixabilityOfStars} shows that multistars are
degree-fixable.  We have also found that the $4$-cycle is degree-fixable.
	
\begin{problem}
Classify the degree-fixable multigraphs (specifically, containment minimal
ones).
\end{problem} 
	
Since $f(v) \ge d_G(v)$, it is convenient to express the values of $f$ as $d+k$
for a non-negative integer $k$; this means $f(v) = d_G(v) + k$.  For brevity,
when $k=0$ we just write $d$, and when $k=1$ we write $d+$, since the figures
only depict the cases $k=0$ and $k=1$.
Looking at the trees in Figures \ref{fig:fixable4}, \ref{fig:fixable5tree}, and
\ref{fig:fixable6tree} we might conjecture that a tree is $f$-fixable whenever
at most one internal vertex is labeled ``$d$''.  This conjecture
continues to hold for many more examples.  
	
\begin{figure}[!htb]
\centering
\includegraphics[scale=0.25]{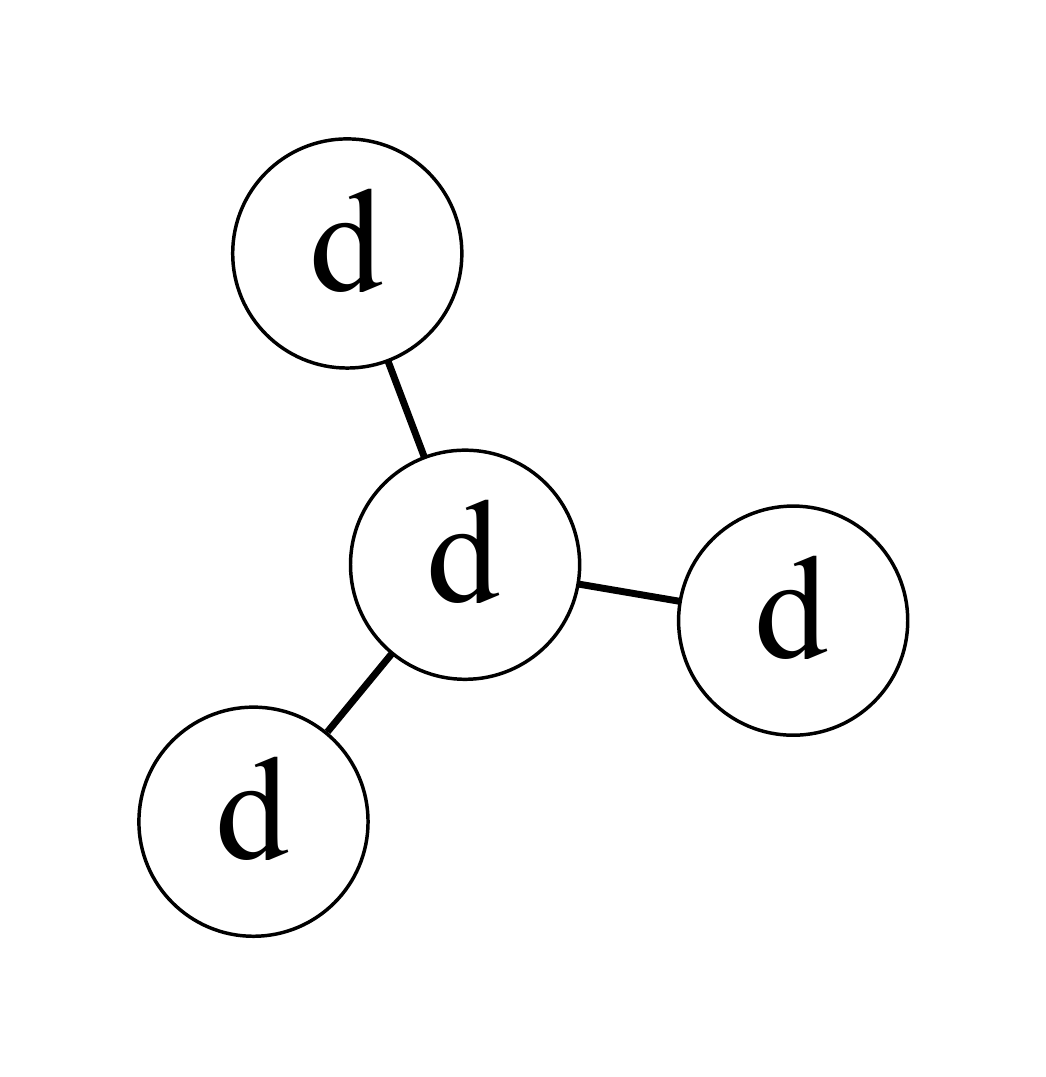}
\includegraphics[scale=0.25]{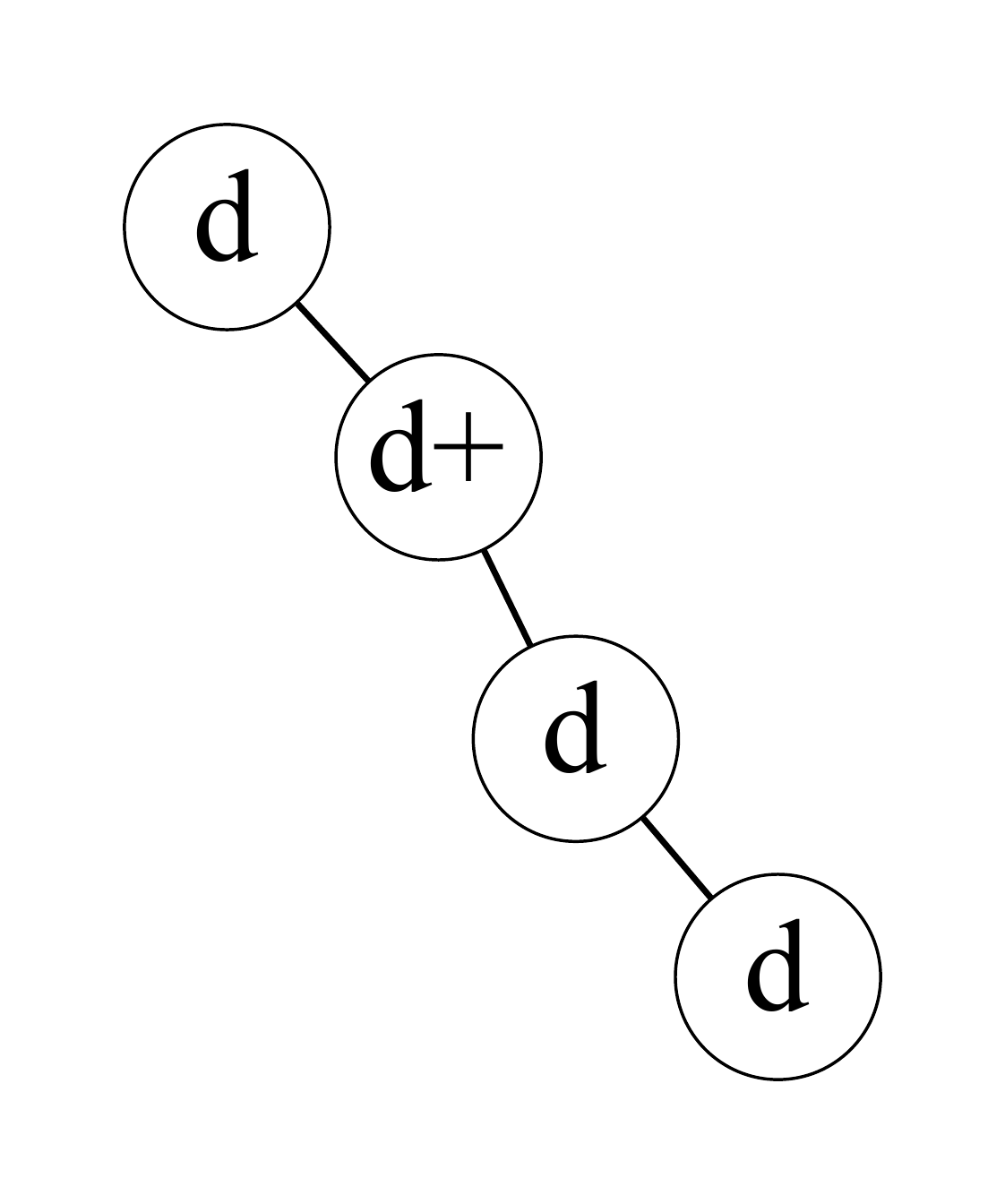}
\includegraphics[scale=0.25]{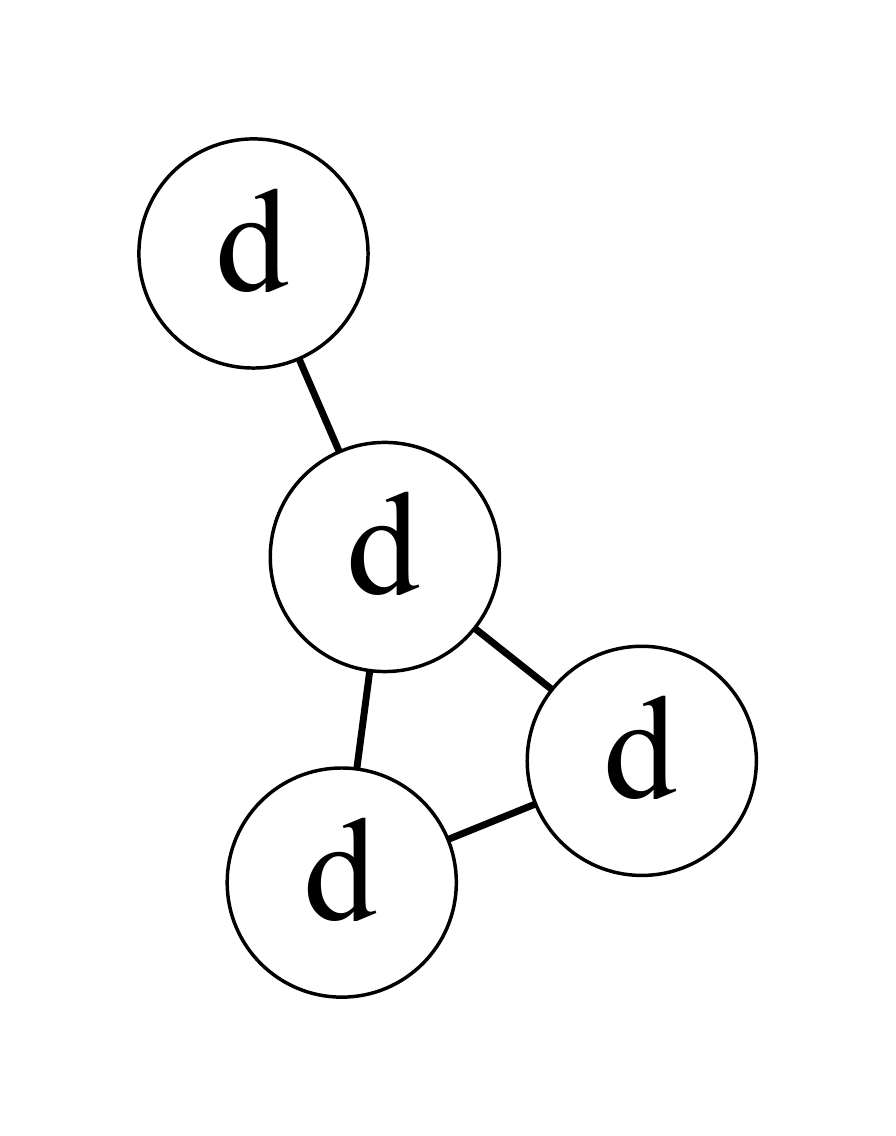}
\includegraphics[scale=0.25]{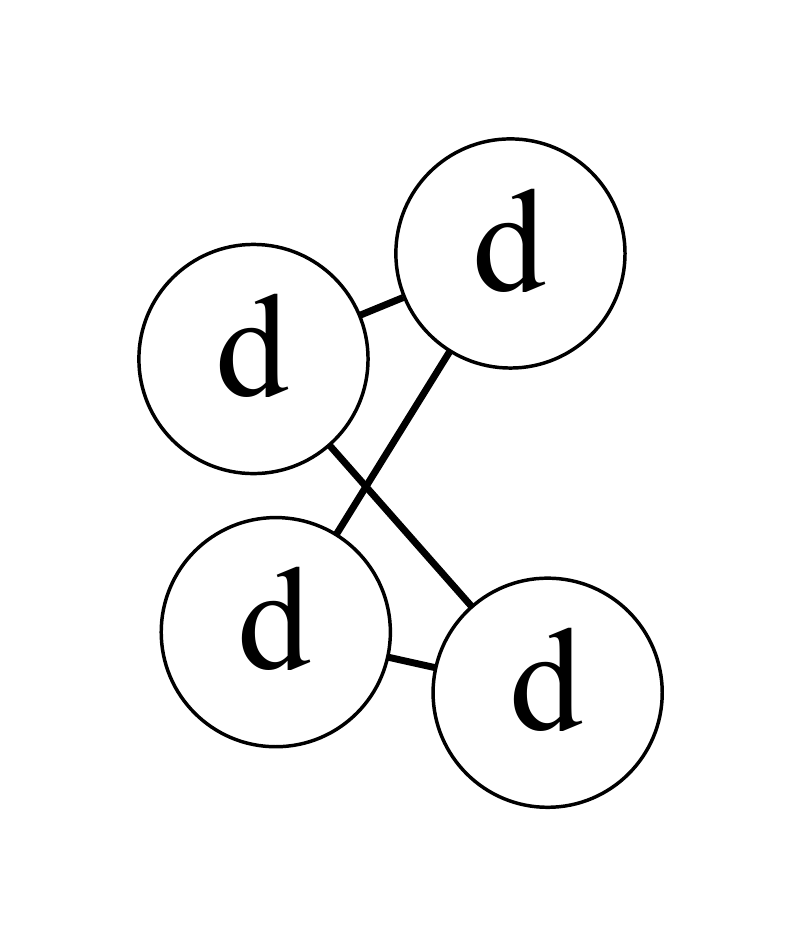}
\includegraphics[scale=0.25]{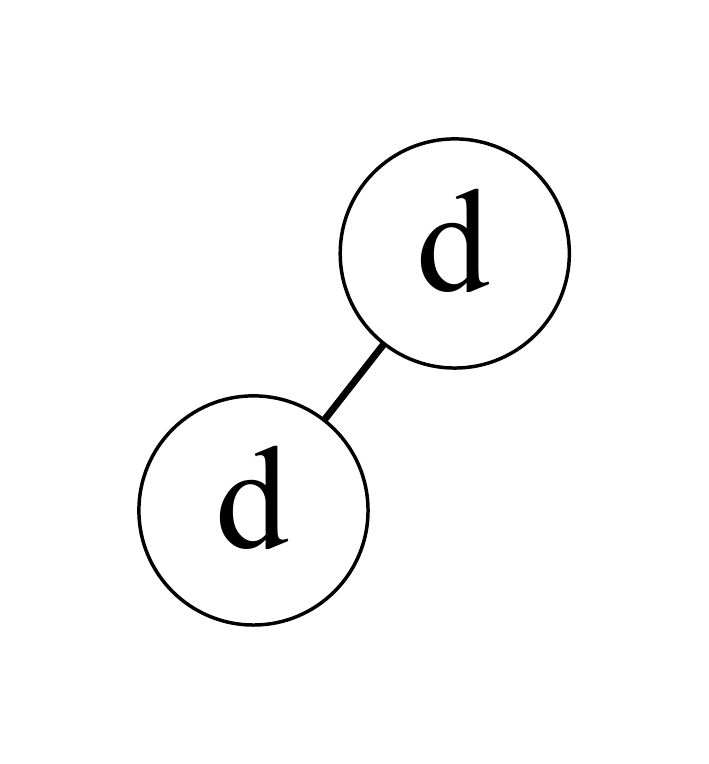}
\includegraphics[scale=0.25]{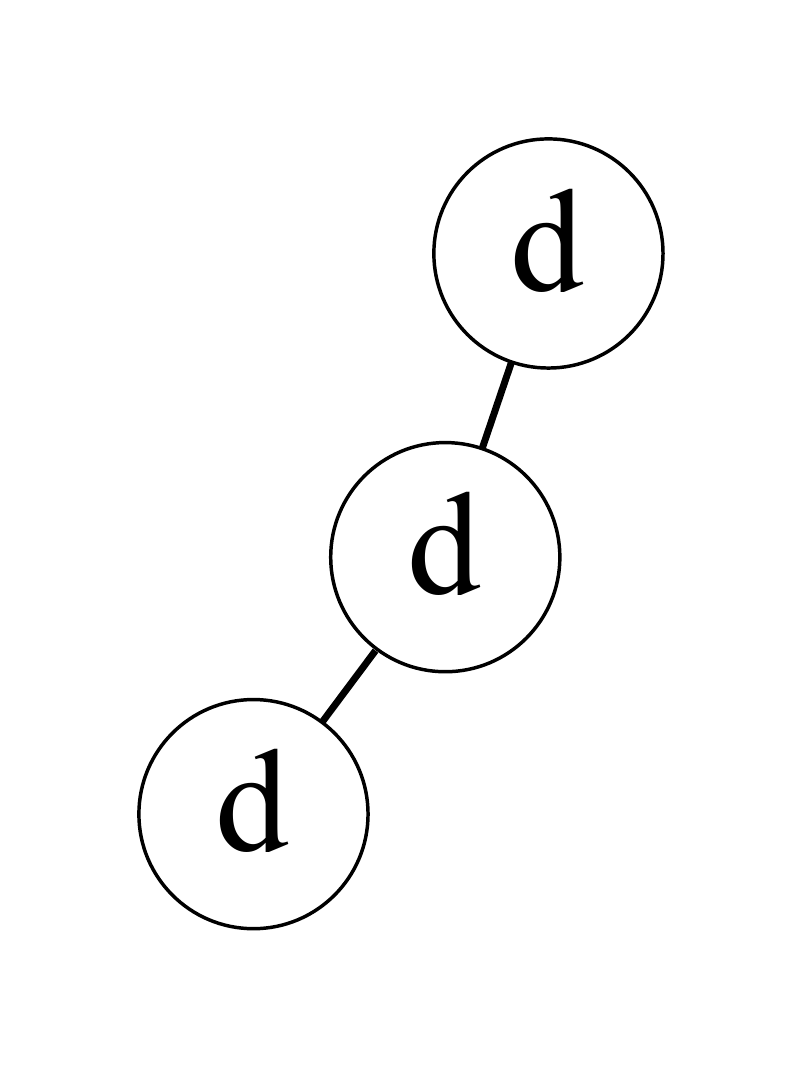}
\includegraphics[scale=0.25]{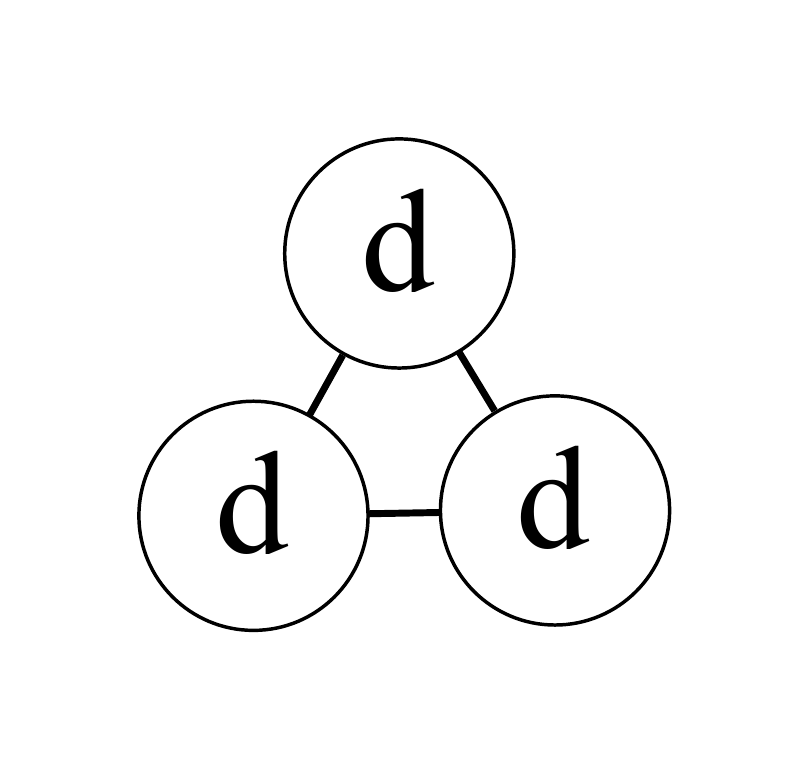}
\includegraphics[scale=0.25]{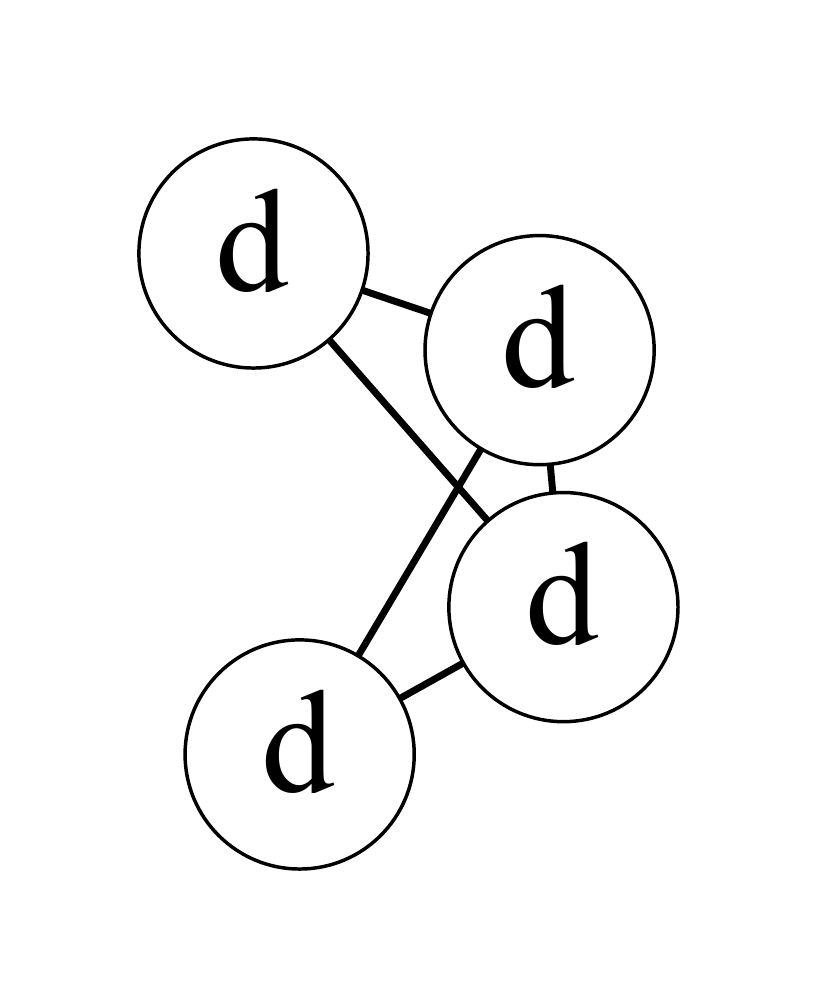}
\includegraphics[scale=0.25]{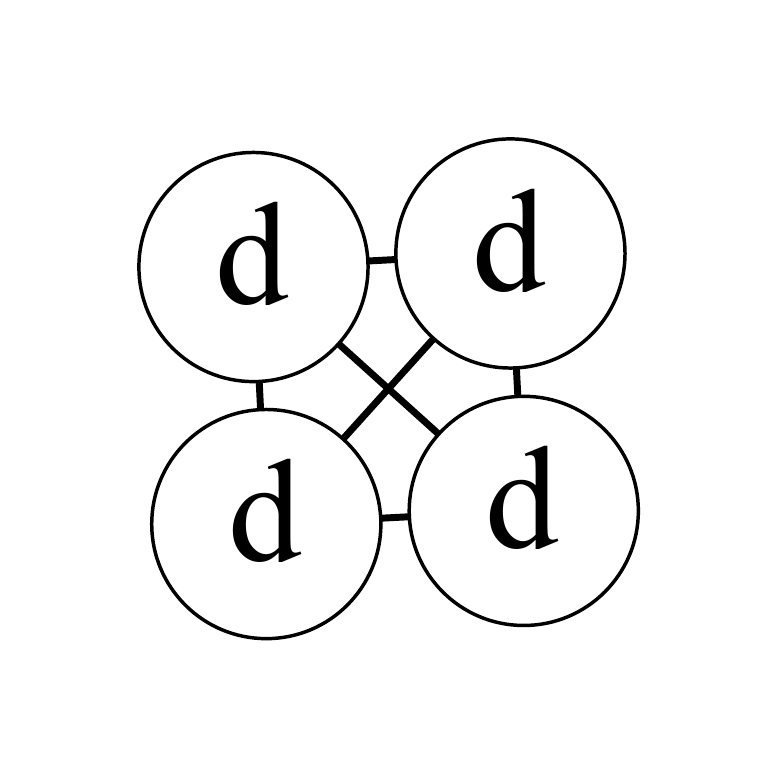}
\caption{The fixable graphs on at most 4 vertices.}
\label{fig:fixable3}
\label{fig:fixable4}
\end{figure}

\begin{figure}[!htb]
\centering
\includegraphics[scale=0.25]{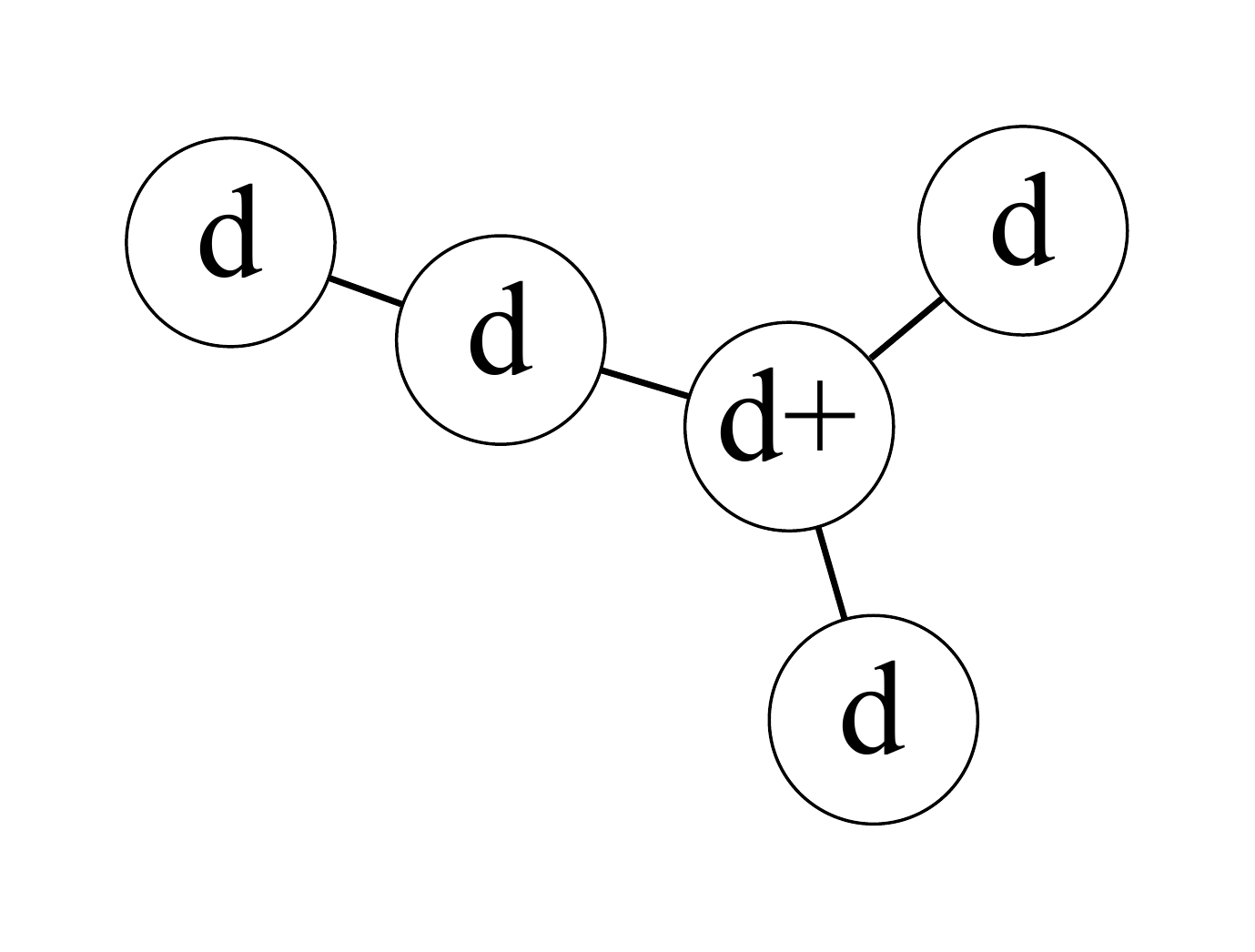}
\includegraphics[scale=0.25]{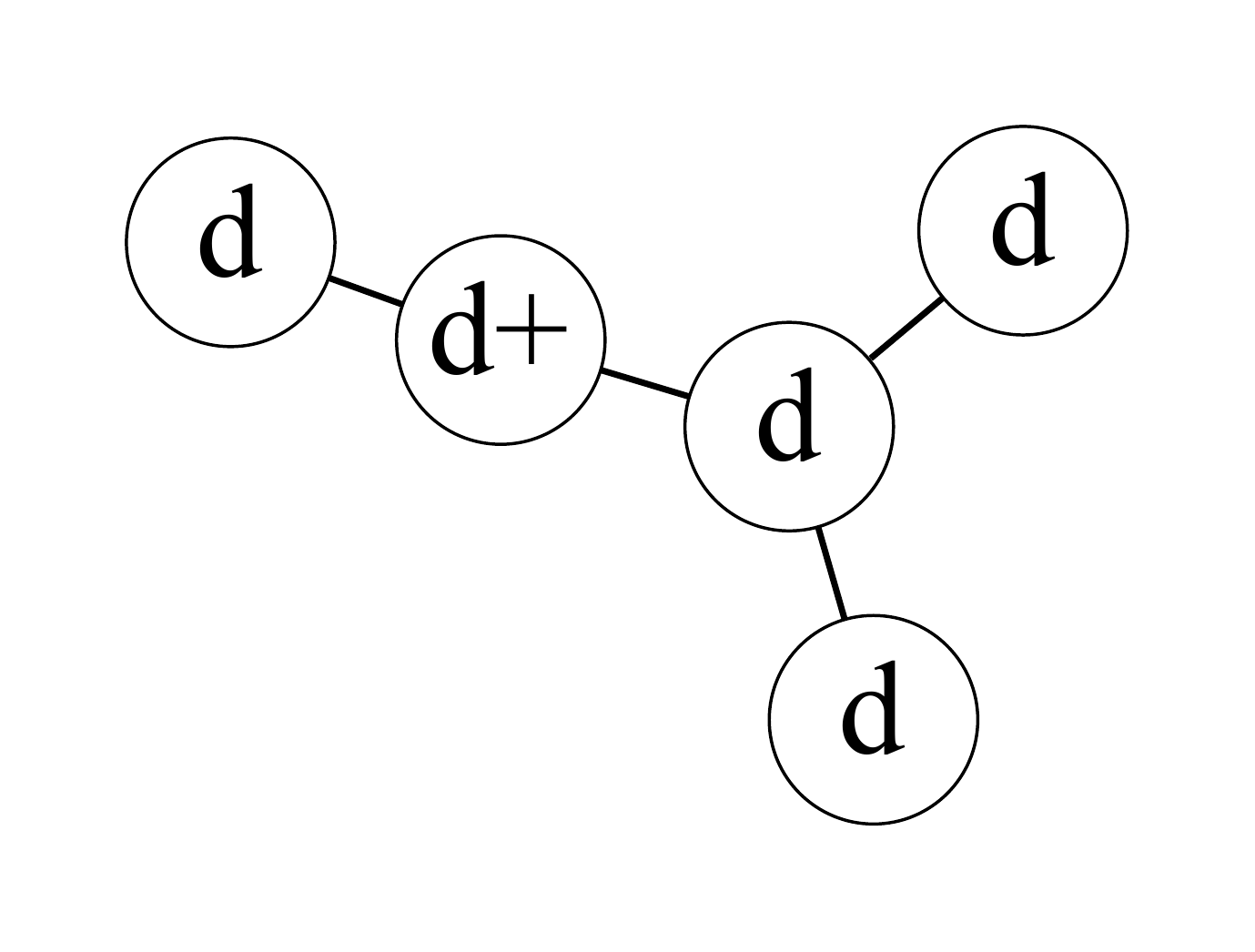}
\includegraphics[scale=0.25]{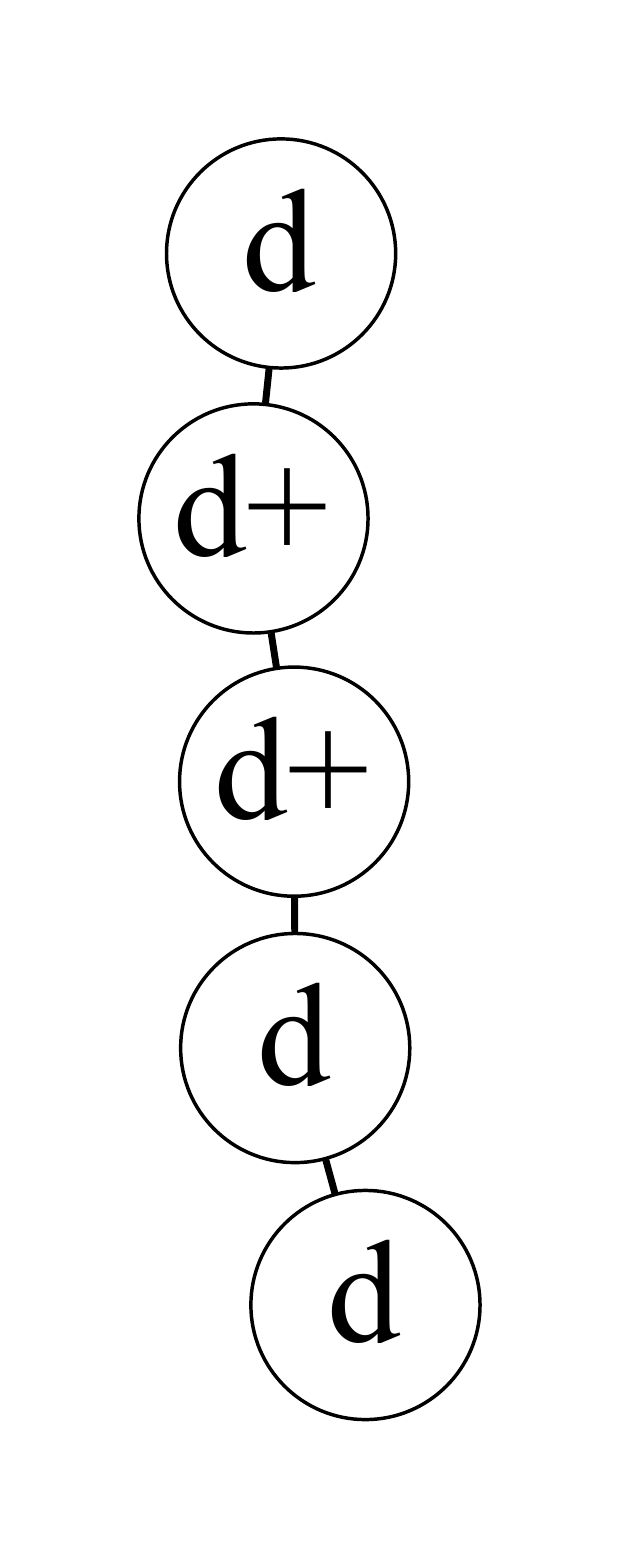}
\includegraphics[scale=0.25]{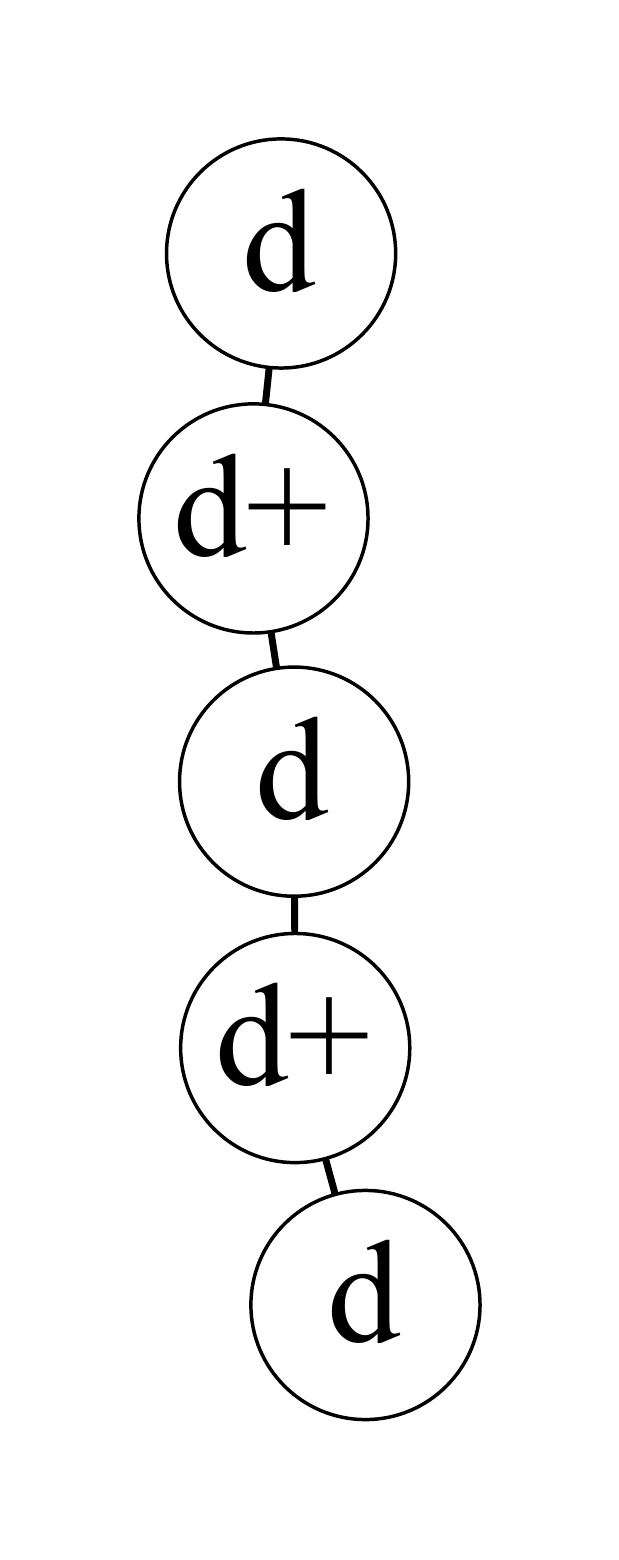}
\caption{The fixable trees with maximum degree at most 3 on 5 vertices.}
\label{fig:fixable5tree}
\end{figure}
		
\begin{figure}[!htb]
\centering
\includegraphics[scale=0.25]{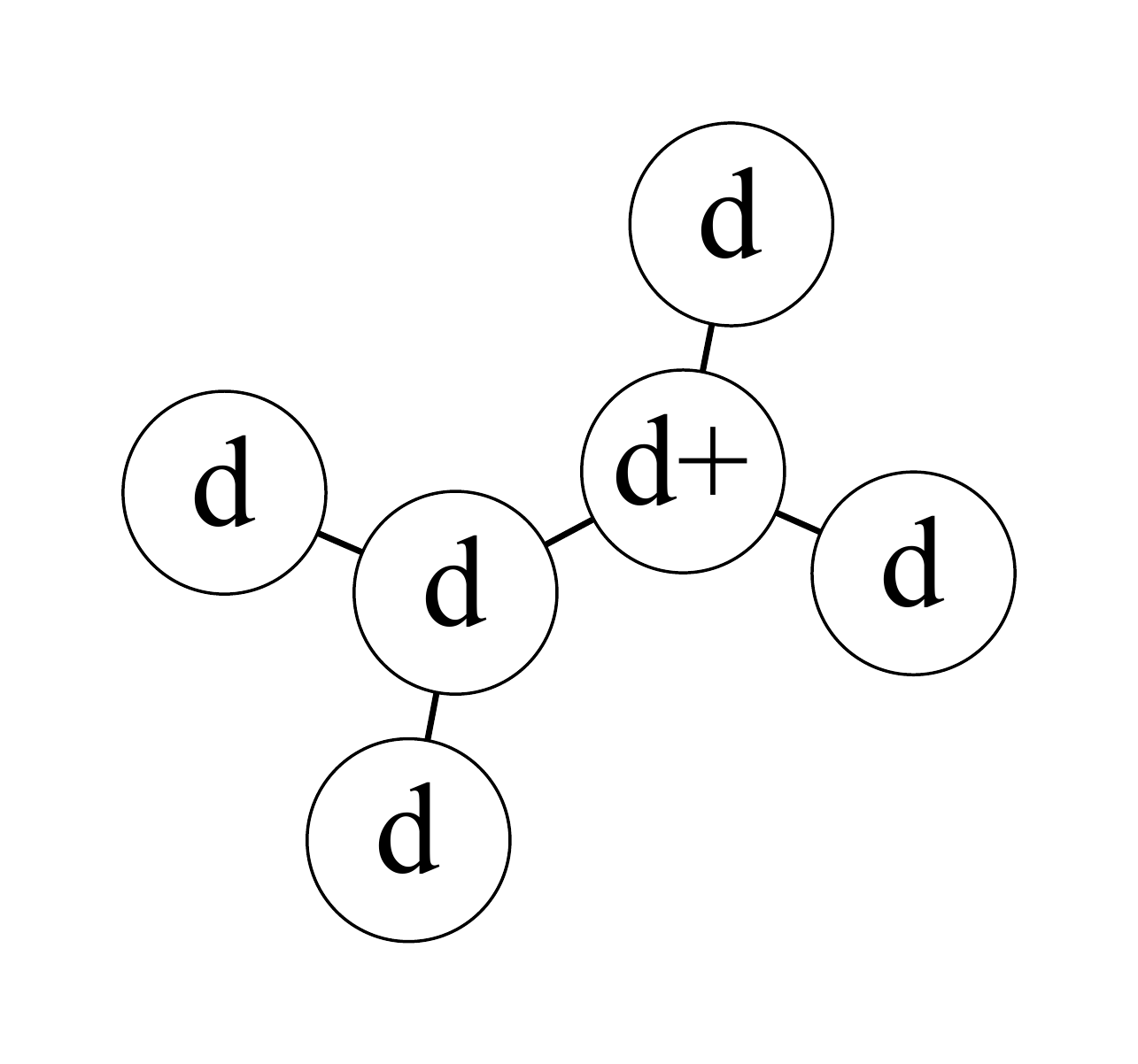}
\includegraphics[scale=0.25]{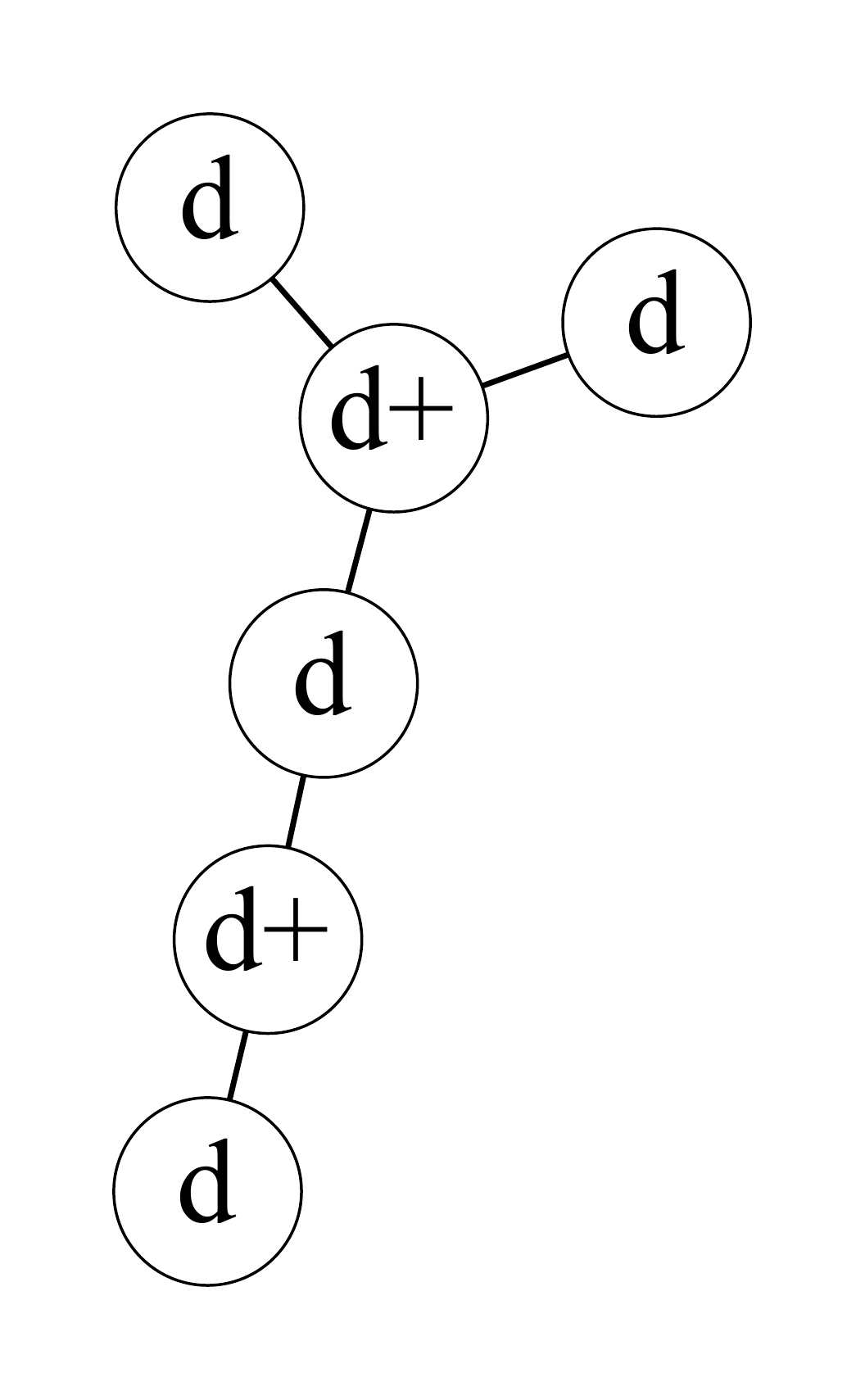}
\includegraphics[scale=0.25]{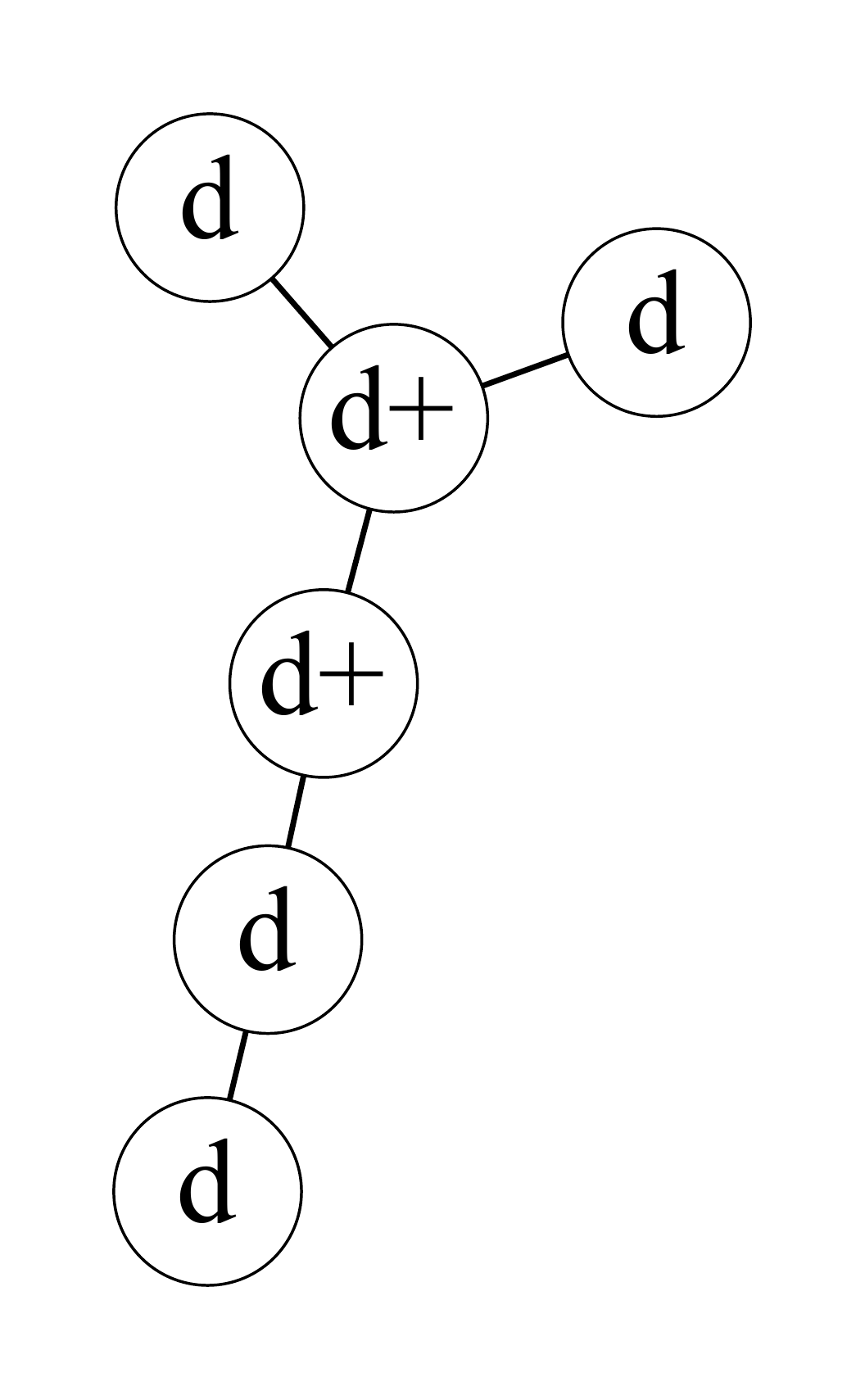}
\includegraphics[scale=0.25]{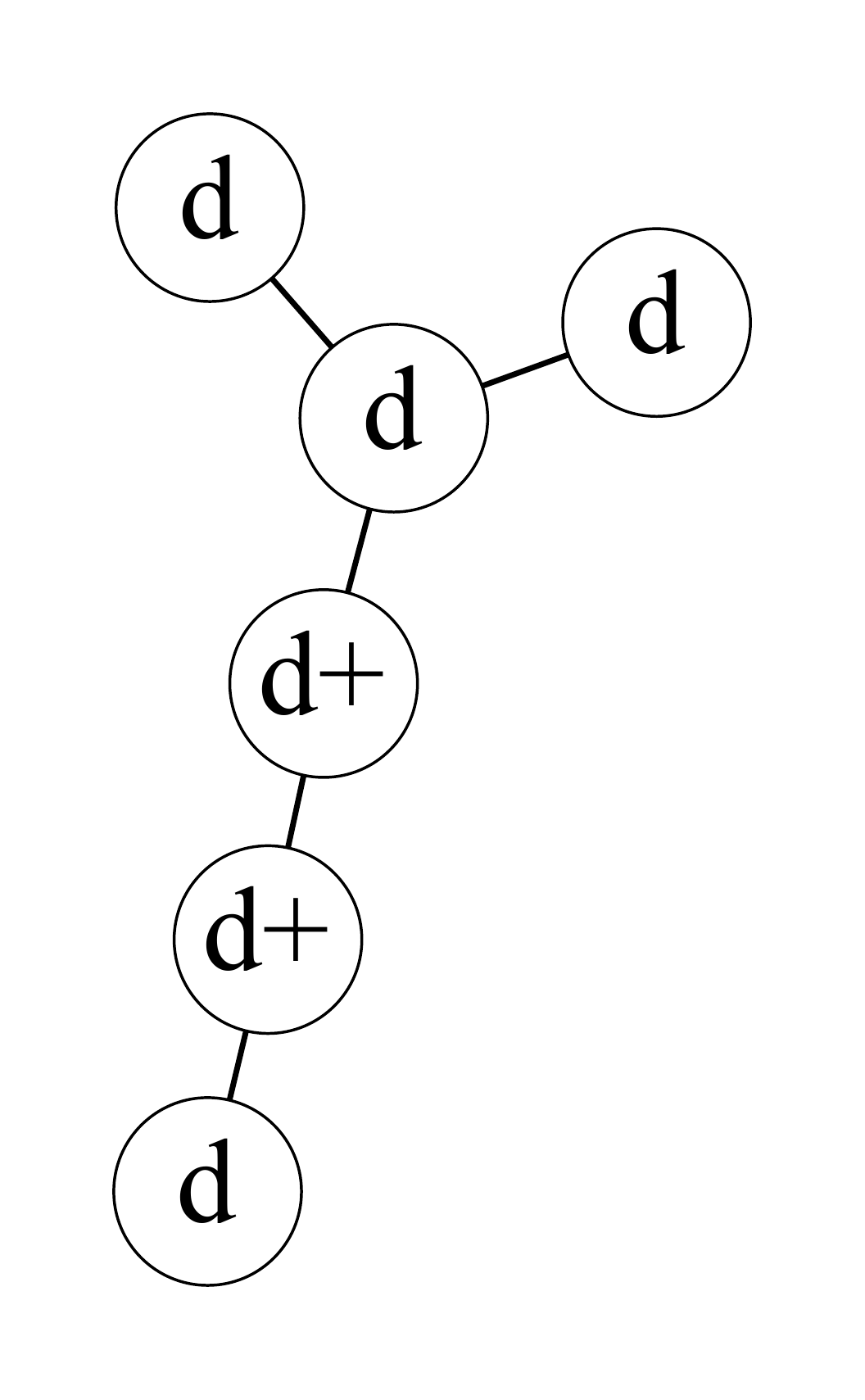}
\includegraphics[scale=0.25]{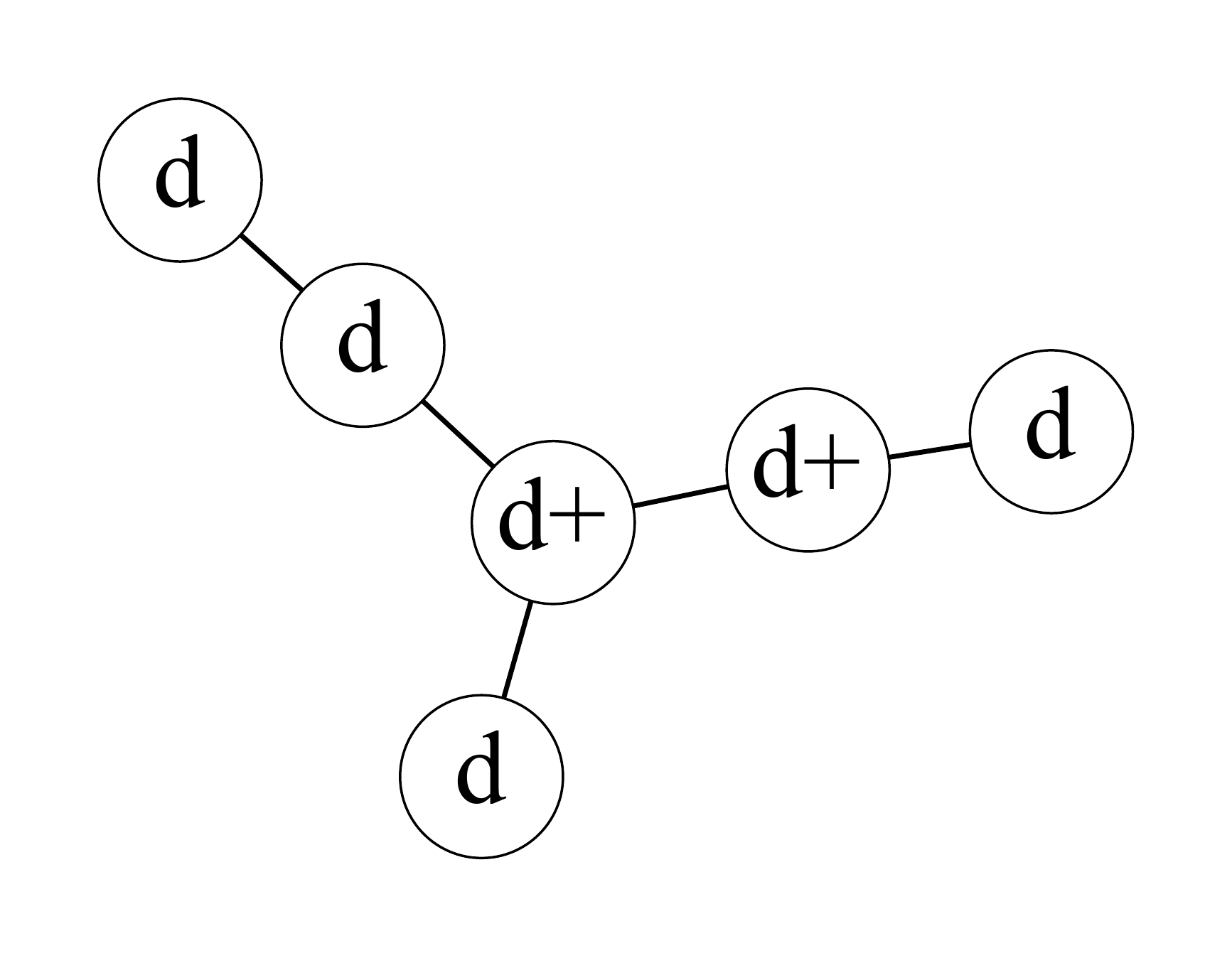}
\includegraphics[scale=0.25]{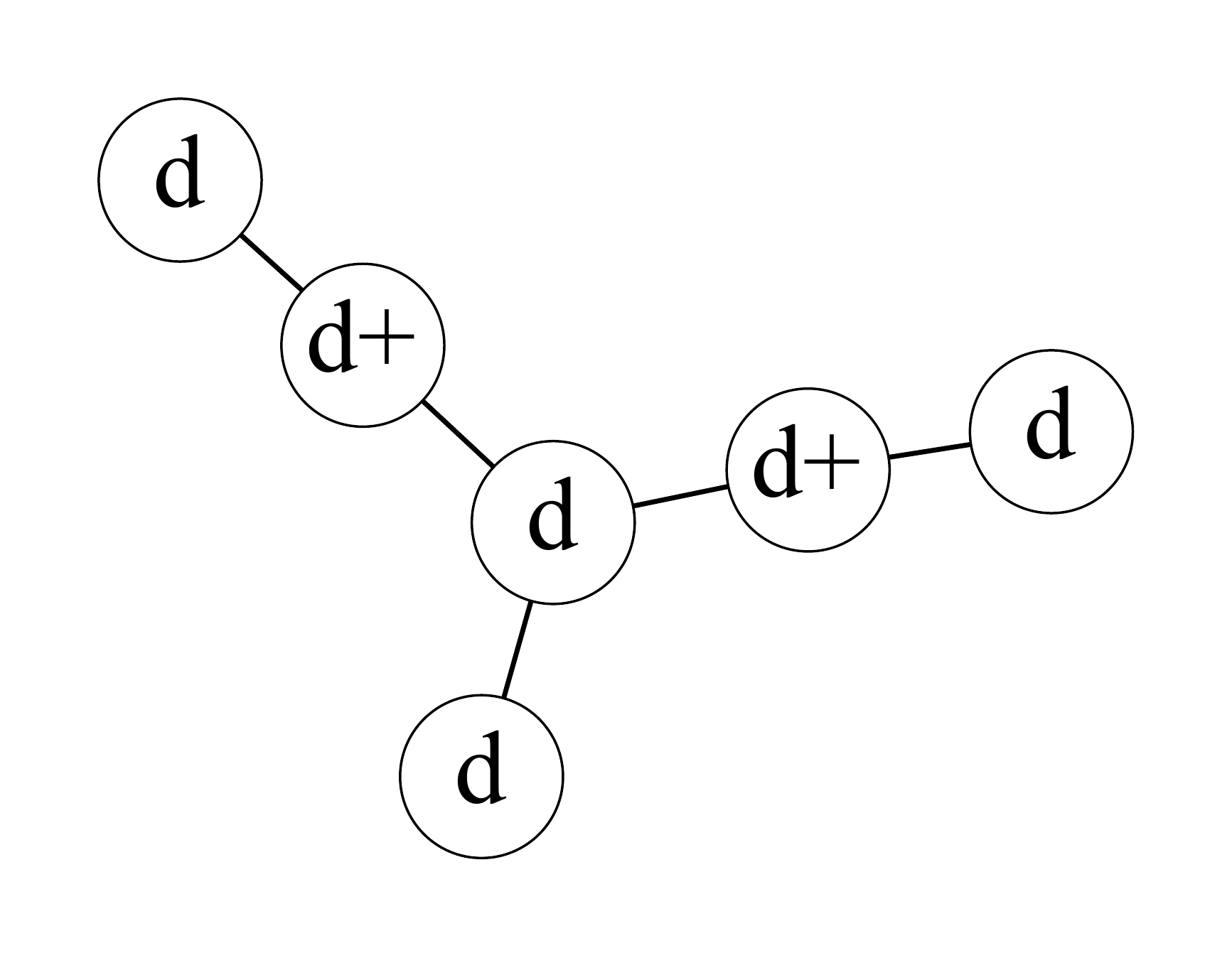}
\includegraphics[scale=0.25]{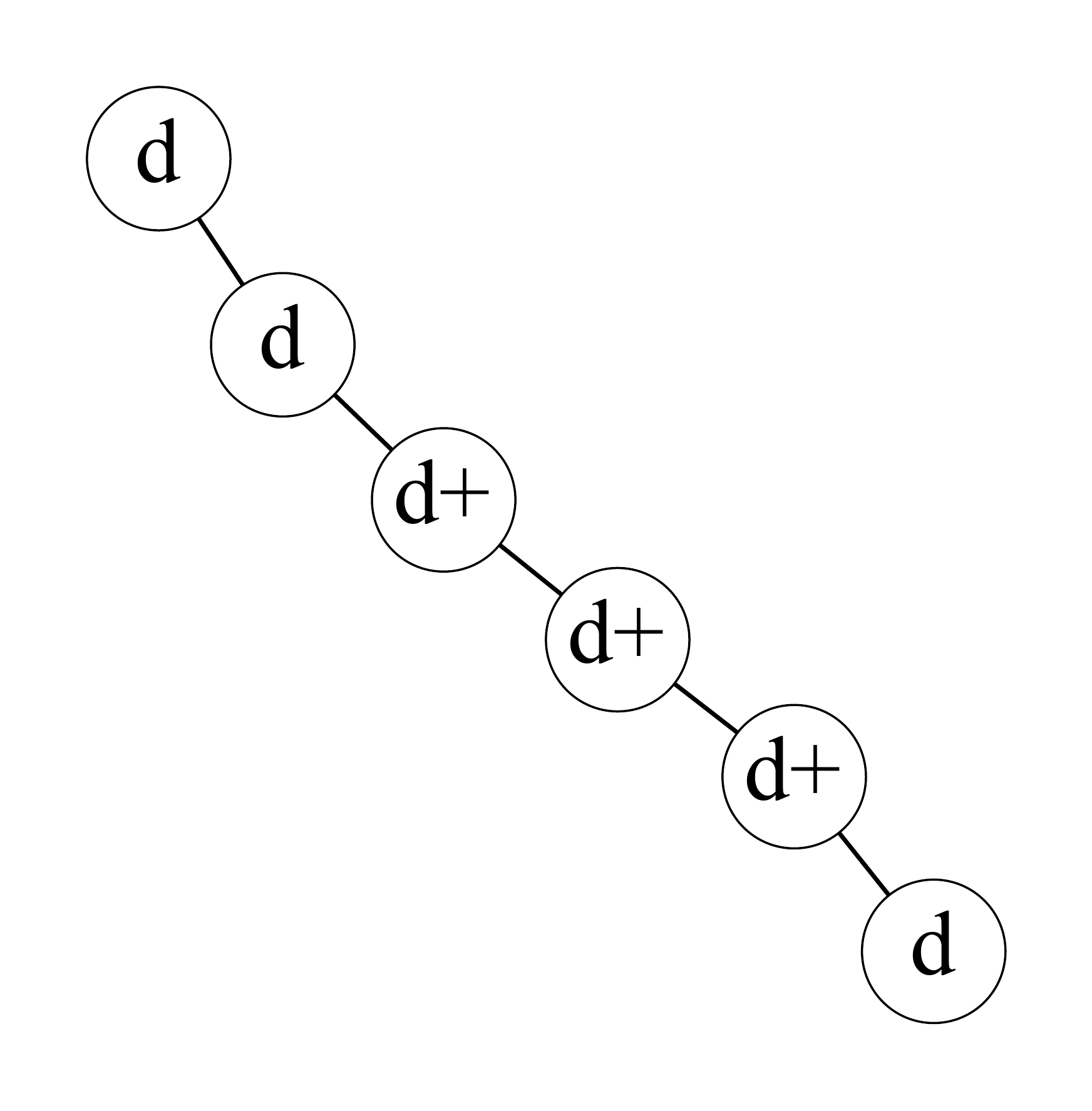}
\includegraphics[scale=0.25]{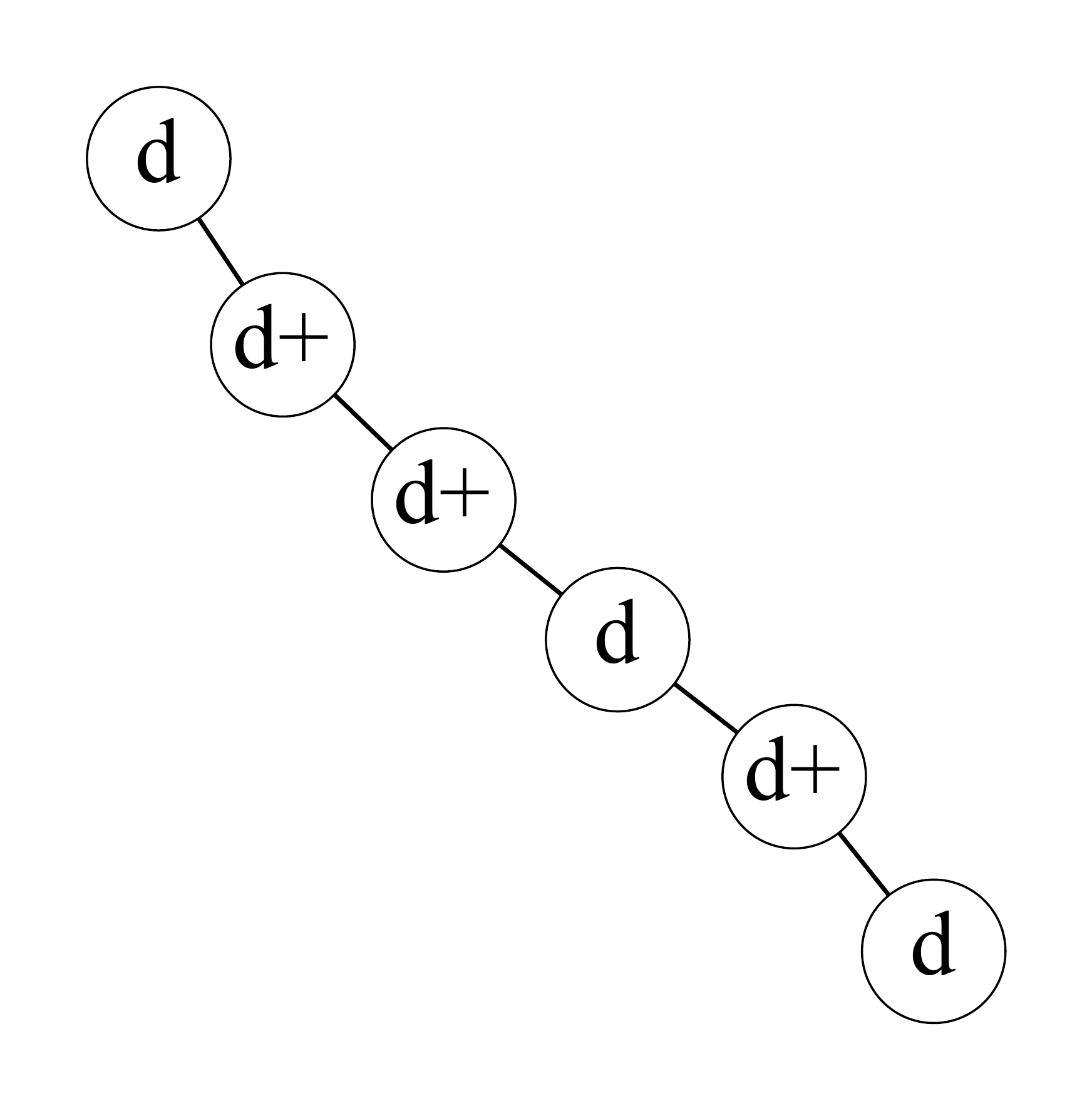}
\caption{The fixable trees with maximum degree at most 3 on 6 vertices.}
\label{fig:fixable6tree}
\end{figure}

\begin{conjecture}
\label{OneHighConjecture}
A tree $T$ is $f$-fixable if $f(v) = d_T(v)$ for at most one non-leaf $v$ of $T$.
\end{conjecture}
Note that by Lemma \ref{KTVImpliesSuperabundant}, this would imply under the
same degree constraints that Tashkinov trees are elementary (that is, each color
is absent from at most one vertex of a Tashkinov tree).  Can this be
proved in the simpler case when the tree is a path?  For paths of length 4, this was
done by Kostochka and Stiebitz; in the next section we conjecture a generalization
of their result to stars with one edge subdivided.   One nice feature of the
superabundance formulation is that since there is no need for an ordering as 
with Tashkinov trees, we can easily formulate results about graphs with cycles.
The following is the most general thing we might think is true.

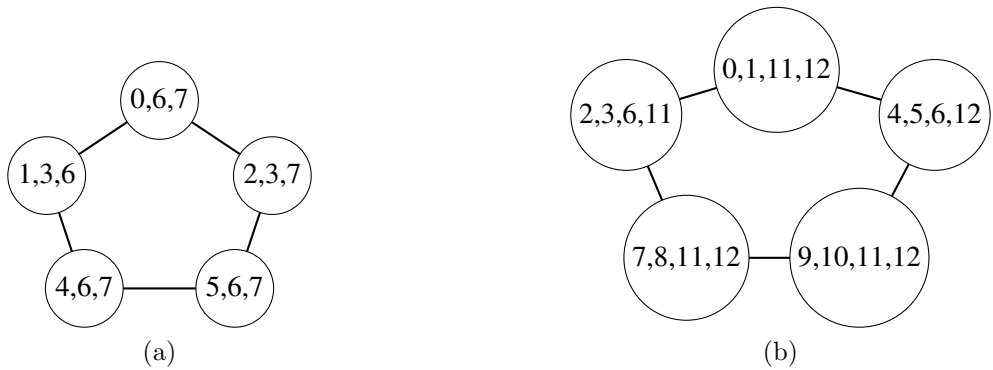
\begin{figure}[!htb]
\renewcommand{\ttdefault}{ptm}
\subfloat[]{\makebox[.5\textwidth]{
\begin{tikzpicture}[scale = 10]
\tikzstyle{VertexStyle} = []
\tikzstyle{EdgeStyle} = []
\tikzstyle{labeledStyle}=[shape = circle, minimum size = 6pt, inner sep = 2.2pt, draw]
\tikzstyle{unlabeledStyle}=[shape = circle, minimum size = 6pt, inner sep = 1.2pt, draw, fill]
\Vertex[style = labeledStyle, x = 0.5, y = 0.90, L = \small {\texttt{0,6,7}}]{v0}
\Vertex[style = labeledStyle, x = 0.35, y = 0.80, L = \small {\texttt{1,3,6}}]{v1}
\Vertex[style = labeledStyle, x = 0.65, y = 0.80, L = \small {\texttt{2,3,7}}]{v2}
\Vertex[style = labeledStyle, x = 0.40, y = 0.65, L = \small {\texttt{4,6,7}}]{v3}
\Vertex[style = labeledStyle, x = 0.60, y = 0.65, L = \small {\texttt{5,6,7}}]{v4}
\Edge[label = \small {}, labelstyle={auto=right, fill=none}](v1)(v0)
\Edge[label = \small {}, labelstyle={auto=right, fill=none}](v2)(v0)
\Edge[label = \small {}, labelstyle={auto=right, fill=none}](v2)(v4)
\Edge[label = \small {}, labelstyle={auto=right, fill=none}](v3)(v1)
\Edge[label = \small {}, labelstyle={auto=right, fill=none}](v3)(v4)
\end{tikzpicture}
}}
\subfloat[]{\makebox[.5\textwidth]{
\begin{tikzpicture}[scale = 10]
\tikzstyle{VertexStyle} = []
\tikzstyle{EdgeStyle} = []
\tikzstyle{labeledStyle}=[shape = circle, minimum size = 6pt, inner sep = 2.2pt, draw]
\tikzstyle{unlabeledStyle}=[shape = circle, minimum size = 6pt, inner sep = 1.2pt, draw, fill]
\Vertex[style = labeledStyle, x = 0.40, y = 0.90, L = \small {\texttt{0,1,11,12}}]{v0}
\Vertex[style = labeledStyle, x = 0.20, y = 0.84, L = \small {\texttt{2,3,6,11}}]{v1}
\Vertex[style = labeledStyle, x = 0.61, y = 0.84, L = \small {\texttt{4,5,6,12}}]{v2}
\Vertex[style = labeledStyle, x = 0.28, y = 0.65, L = \small {\texttt{7,8,11,12}}]{v3}
\Vertex[style = labeledStyle, x = 0.51, y = 0.65, L = \small {\texttt{9,10,11,12}}]{v4}
\Edge[label = \small {}, labelstyle={auto=right, fill=none}](v1)(v0)
\Edge[label = \small {}, labelstyle={auto=right, fill=none}](v2)(v0)
\Edge[label = \small {}, labelstyle={auto=right, fill=none}](v2)(v4)
\Edge[label = \small {}, labelstyle={auto=right, fill=none}](v3)(v1)
\Edge[label = \small {}, labelstyle={auto=right, fill=none}](v3)(v4)
\end{tikzpicture}
}}
\caption{
Counterexamples to Conjecture \ref{MoonshineConjecture}; the numbers at each
vertex are its list of available colors.
}
\label{fig:goldbergce}
\end{figure}

\begin{conjecture}[false]
	A multigraph $G$ is $f$-fixable if $f(v) > d_G(v)$ for all $v \in V(G)$.
	\label{MoonshineConjecture}
\end{conjecture}

This conjecture is very strong and implies Goldberg's conjecture 
(see \cite[p.~155ff]{stiebitz2012graph}), which is one of the major open
problems in edge-coloring.
Unfortunately, Conjecture~\ref{MoonshineConjecture} is false.  We can make
counterexamples on a $5$-cycle as in Figure \ref{fig:goldbergce}.  We don't yet
have an intuitive explanation for why these are counterexamples, but in each
case the computer has found a strategy preventing the 5-cycle from being colored.

One interesting consequence of these counterexamples 
is that $C_5$ is not $f$-fixable for \emph{any}
function $f$.  (This contrasts with the case of $(f,k)$-fixable, since now
increasing $f$ need not increase $\psi(L)$.)
Let $L$ denote the lists in Figure~\ref{fig:goldbergce}(b).
Given a function $f$, begin with $L$ and add as many ``singletons'' (colors that appear in
only one list) as needed to the lists so that each list is large enough; call
these lists $L'$.  Since $L$ is superabundant, clearly so is $L'$.
Now we play an ``extra'' game against the computer using $L$ and we use the
computer's stratey in this extra game to inform our strategy for $L'$ in the
real game.  If the colors chosen to swap in the real game are both in $L$, then
we play with the computer's strategy for the extra game.  Any color $c\in
L'\setminus L$ at a vertex $v$ in the extra game is a singleton.  Since only
three colors of $L$ are non-singletons, and $|L(v)|=4$, each vertex in the
extra game will always have a singleton in its list.  Thus, we treat $c$ like
some other singleton $c'$ at $v$ in the extra game, and use the computer's strategy
from the extra game if $c'$ had been chosen instead.

We have more questions than answers about Conjectures~\ref{OneHighConjecture}
and~\ref{MoonshineConjecture}.  For instance,
what if in Conjecture~\ref{MoonshineConjecture} we only look at
superabundant list assignments arising from an edge coloring of $G-e$, for some
edge $e$?  The resulting conjecture is also stronger than Goldberg's
conjecture, and at present we have no counterexamples.
	
\subsection{Stars with one edge subdivided}
The following conjecture would generalize the ``Short Kierstead Paths'' of Kostochka
and Stiebitz (see \cite[p.~46ff]{stiebitz2012graph}).  Parts (a) and (b) are
special cases of Conjecture \ref{OneHighConjecture}.  We have a rough draft of
a proof for part (a) and we suspect parts (b) and (c) will be similar, 
but our draft is long and detailed, and we are still hoping to find a clean
proof, like that for stars.
Recall that the fan equation implies the reducibility for $k$-edge-coloring of
stars with certain specified degrees of the leaves.  In this section we show
that the truth of Conjecture \ref{StarWithOneEdgeSubdivided} would imply a
similar equation for stars with one edge subdivided.

\begin{conjecture}
\label{StarWithOneEdgeSubdivided}
Let $G$ be a star with one edge subdivided, where $r$ is the center of the
star, $t$ the vertex at distance two from $r$, and $s$ the intervening vertex.  
If $L$ is superabundant and $|L(v)| \ge d_G(v)$ for all 
$v \in V(G)$, then $G$ is $L$-fixable if at least one of the following holds:
\begin{enumerate}
\item[(a)] $|L(r)| > d_G(r)$; or
\item[(b)] $|L(s)| > d_G(s)$; or
\item[(c)] $\psi_L(G) > \size{G}$.
\end{enumerate}
\end{conjecture}

For a graph $H$ and $v \in V(H)$, let $E_H(v)$ be the set of edges incident to
$v$ in $H$. Let $Q$ be an edge-critical graph with $\chi'(Q) = \Delta(Q) + 1$
and $G \subseteq Q$.  For a $\Delta(Q)$-edge-coloring $\pi$ of $Q - E(G)$, let
$L_\pi(v) = \irange{\Delta(Q)} - \pi\parens{E_Q(v) - E(G)}$ for all $v \in
V(G)$.  Graph $G$ is a \emph{$\Psi$-subgraph} of $Q$ if there is a
$\Delta(Q)$-edge-coloring $\pi$ of $Q - E(G)$ such that each $H \subsetneq G$
is abundant. Let $E_{L}(H) = \card{\setb{\alpha}{\pot(L)}{\card{H_{L, \alpha}}
\text{ is even}}}$ and $O_{L}(H) = \card{\setb{\alpha}{\pot(L)}{\card{H_{L,
\alpha}} \text{ is odd}}}$.  Note that $\pot(L) = E_{L}(G) + O_{L}(G)$.

\begin{lem}\label{LowPsiGivesManyOddColors}
Let $Q$ be an edge-critical graph with $\chi'(Q) = \Delta(Q) + 1$. If $G
\subseteq Q$ and $\pi$ is a $\Delta(Q)$-edge-coloring of $Q - E(G)$ such that
$\size{G}\ge \psi_L(G) $, then $\card{O_{L_\pi}(G)} \ge \sum_{v \in V(G)}
\Delta(Q) - d_Q(v)$.  Furthermore, if $\size{G} > \psi_L(G)$, then
$\card{O_{L_\pi}(G)} > \sum_{v \in V(G)} \Delta(Q) - d_Q(v)$.
\end{lem}
\begin{proof}
The proof is a straightforward counting argument.  For fixed degrees and list
sizes, as $\card{O_L(G)}$ gets larger, $\psi_L(G)$ gets smaller (half as
quickly).  The details forthwith.  Let $L = L_\pi$.

Since $\size{G} \ge \psi_L(G)$, we have 

\begin{align}
\label{edge-crit1}
\size{G} \ge 
\sum_{\alpha \in \pot(L)} \floor{\frac{\card{G_{L, \alpha}}}{2}}  =
\sum_{\alpha \in \pot(L)} \frac{\card{G_{L, \alpha}}}{2} -  \sum_{\alpha \in
	O_L(H)} \frac12 
.\end{align}
	
\noindent 
Also,

\begin{align}
\sum_{\alpha \in \pot(L)} \frac{\card{G_{L, \alpha}}}{2} 
&= \sum_{v \in V(G)} \frac{\Delta(Q) - (d_Q(v)-d_G(v))}{2} \notag\\
&= \sum_{v \in V(G)} \frac{d_G(v)}{2} + \sum_{v \in V(G)} \frac{\Delta(Q) -
d_Q(v)}{2}\notag\\
&= \size{G} +  \sum_{v \in V(G)} \frac{\Delta(Q) - d_Q(v)}{2}.
\label{edge-crit2}
\end{align}
	
\noindent Now we solve for $\size{G}-
\sum_{\alpha \in \pot(L)} \frac{\card{G_{L, \alpha}}}{2}$ in 
\eqref{edge-crit1} and \eqref{edge-crit2}, set the expressions equal, and then
simplify.  The result is \eqref{edge-crit3}.
	
\begin{align}
\card{O_L(G)} \ge \sum_{v \in V(G)} \Delta(Q) - d_Q(v).
\label{edge-crit3}
\end{align}
Finally, if the inequality in \eqref{edge-crit1} is strict, then the inequality
in \eqref{edge-crit3} is also strict.
\end{proof}

Again, let $Q$ be an edge-critical graph with $\chi'(Q) = \Delta(Q) + 1$ and $G
\subseteq Q$.  If there is a $\Delta(Q)$-edge-coloring $\pi$ of $Q - E(G)$ such
that each $H \subsetneq G$ is abundant, then $G$ is a \emph{$\Psi$-subgraph} of
$Q$.  The point of this definition is that if $G$ is a $\Psi$-subgraph (and
Conjecture~\ref{StarWithOneEdgeSubdivided}(c) holds), then $\size{G}\ge
\psi(G)$, so we can apply Lemma~\ref{LowPsiGivesManyOddColors}.

\begin{conjecture}\label{AdjacencyPrecursor}
Let $Q$ be an edge-critical graph with $\chi'(Q) = \Delta(Q) + 1$.  
Let $H$ be a star with one edge subdivided; let $r$ be the center of the star,
$t$ the vertex at distance two from $r$, and $s$ the intervening vertex. 
If $H$ is a $\Psi$-subgraph of $Q$, 
then there exists $X \subseteq N(r)$ with $V(H - r - t)
\subseteq X$ such that \[\sum_{v \in X \cup \set{t}} (d_Q(v) + 1 - \Delta(Q))
\ge 0.\]  
	
\noindent Moreover, if $\set{r,s,t}$ does not induce a triangle in $Q$, then 
\[\sum_{v \in X \cup \set{t}} (d_Q(v) + 1 - \Delta(Q)) \ge 1.\]
Furthermore, if $d_Q(r)<\Delta(Q)$ or $d_Q(s)<\Delta(Q)$, then both lower
bounds improve by 1.
\end{conjecture}
\begin{proof}[Proof (assuming Conjecture~\ref{StarWithOneEdgeSubdivided}).]
Let $G$ be a maximal $\Psi$-subgraph of $Q$ containing $H$ such that $G$
is a star with one edge subdivided.  Let $\pi$ be a coloring of $Q - E(G)$
showing that $G$ is a $\Psi$-subgraph and let $L = L_\pi$.  
	
We first show that $\card{E_{L}(G)} \ge d_Q(r) - d_G(r) - 1$ if $rst$ induces a
triangle; otherwise, $\card{E_{L}(G)} \ge d_Q(r) - d_G(r)$.
Suppose $rst$ does not induce a triangle; for an arbitrary $x \in N_Q(r) - V(G)$,
let $\alpha=\pi(rx)$.  Now consider adding $x$ to $G$.  
By assumption, every $J\subsetneq G$ is abundant.  Further, if $J\subsetneq G$
is abundant, then $J+x$ is also abundant.  Thus, we only need to show that $G$
is abundant.  If $\alpha \in O_{L}(G)$, then adding $x$ to $G$ makes $G$
abundant, since now $r$ also has $\alpha$ in its list.  This gives a larger
$\Psi$-subgraph of the required form, which contradicts the maximality of
$G$.  Hence $\alpha \in E_{L}(G)$.  Therefore, $\card{E_{L}(G)} \ge d_Q(r) -
d_G(r)$ as desired.  If $rst$ induces a triangle, then we lose one off this
bound from the edge $rt$.
	
By Conjecture \ref{StarWithOneEdgeSubdivided}(c), we have $\psi_L(G) \le
\size{G}$.  Hence, by Lemma \ref{LowPsiGivesManyOddColors}, we have
$\card{O_{L}(G)} \ge \sum_{v \in V(G)} \Delta(Q) - d_Q(v)$.  If $rst$ does
not induce a triangle, then
	
\begin{align*}
\Delta(Q) &\ge \pot(L)\\
&= \card{E_{L}(G)} + \card{O_{L}(G)}\\
&\ge d_Q(r) - d_G(r) + \sum_{v \in V(G)} \Delta(Q) - d_Q(v) \numberthis
\label{strict-ineq}\\
&= \Delta(Q) - d_G(r) + \sum_{v \in V(G - r)} \Delta(Q) - d_Q(v)\\
&= \Delta(Q) + 1 +\sum_{v \in V(G - r)} \Delta(Q) - 1 - d_Q(v).
\end{align*}
	
Therefore, $\sum_{v \in V(G - r)} \Delta(Q) - 1 - d_Q(v) \le -1$.  Negating
gives the desired inequality.  If $rst$ induces a triangle, then we lose one
off the bound.  Conjecture \ref{StarWithOneEdgeSubdivided}(a,b) gives the final
statement.
\end{proof}

\section{Algorithm Overview}
Here we describe the basic outline of our algorithm to test if a given graph $G$ is $k$-fixable.  To test if $G$ is $(L,P)$-fixable for one $L$, we need to generate the two-player game tree.  
Doing this for every $L$ would be a lot of work.  With memoization, we can cut this down and get a reasonably efficient algorithm, but we can do much better by changing to a bottom-up strategy; that is, we do dynamic programming as follows. 

\begin{enumerate}
	\item Generate the set $\mathcal{L}$ of all possible lists assignments $L$ on $G$ with $\pot{(L) \subseteq \irange{k}}$.
	\item Create a set $\mathcal{W}$ of \emph{won} assignments, consisting of all $L \in \mathcal{L}$ such that $G$ is $L$-colorable.
	\item Put $\mathcal{L} \DefinedAs \mathcal{L} \setminus \mathcal{W}$.
	\item For each $L \in \mathcal{L}$, check if there are different colors $a,b \in \irange{k}$ such that for every partition
	$X_1, \ldots, X_t$ of $S_{L,a,b}$ into sets of size at most two, there exists $J
	\subseteq \irange{t}$ so that $L' \in \mathcal{W}$, where $L'$ is formed
	from $L$ by swapping $a$ and $b$ in $L(v)$ for every $v \in \bigcup_{i \in J} X_i$.  If so, add $L$ to $\mathcal{W}$.
	\item If step $(4)$ modified $\mathcal{W}$, goto step (3).
	\item $G$ is $k$-fixable if and only if $\mathcal{L} = \emptyset$.
\end{enumerate}

In step (1), we do not really want to generate \emph{all} list assignments, just list assignments up to color permutation.  To do this generation, we put an ordering on the set of list assignments and run an algorithm that outputs only the minimal representative of each color-permutation class.  All the code lives in the GitHub repository of WebGraphs at \url{https://github.com/landon/WebGraphs}.   Since a lot of this code is optimized for speed and not readability, a reference version is currently being built at \url{https://github.com/landon/Playground/tree/master/Fixability}.

\section{Conclusion}
Most work on proving sufficient conditions for $k$-edge-colorability relies on
proving that various configurations are reducible.  Although these reducibility
proofs have common themes, they often feel ad hoc and are tailored to the
specific theorem being proved.  We have introduced the notion of fixability,
which provides a unifying framework for many of these results.  It also
naturally leads to a number of conjectures which, if true, will likely increase 
greatly what we can prove about $k$-edge-coloring.  The computer has provided
significant experimental evidence for these conjectures (proving many specific
cases), but offers little guidance toward proving them completely.

To conclude, we mention two consequences if
Conjectures~\ref{StarWithOneEdgeSubdivided} and~\ref{OneHighConjecture}
are true.
Vizing~\cite{vizing68unsolved} conjectured that every $\Delta$-critical graph
has average degree
greater than $\Delta-1$.  For large $\Delta$, the best lower bound is about
$\frac23\Delta$, due to Woodall~\cite{woodall2007average}. 
His proof relies on a new class of reducible configurations, which would be
implied if both $P_5$ is fixable (a very special case of
Conjecture~\ref{OneHighConjecture}) and
Conjecture~\ref{StarWithOneEdgeSubdivided} is true.  Another old conjecture of
Vizing~\cite{vizing65chromatic} is that every $n$-vertex graph has independence
number at most $\frac12n$.  The best upper bound known is $\frac35n$,
also due to Woodall~\cite{woodall2011independence}.  The proof is relatively
short, but relies on the same reducible configurations just mentioned.  Thus, proving
Conjecture~\ref{AdjacencyPrecursor} and
Conjecture~\ref{StarWithOneEdgeSubdivided} (even just for $P_5$) would put the
best bounds for these two old problems into a much broader context.


%

\bibliographystyle{amsplain}
\bibliography{GraphColoring}

\clearpage
\section*{Appendix: Impoved lower bound on the average degree of 4-critical graphs}
Here we prove the $\Delta=4$ case of Woodall's conjecture
\cite{woodall2008average} on the average degree of a critical graph (modulo
computer proofs of reducibility).  We show all of the reducible configurations
in Figures~\ref{fig:C1}--\ref{fig:C39}.  A \emph{board} is a list assignment for
the vertices of the configuration.  A board is \emph{colorable} if the
configuration can be colored immediately from that board.  The computer shows
that a configuration is reducible by considering all possible boards and
verifying for each that some sequence of Kempe swaps leads to a colorable board.
The \emph{depth} of a board is the minimum number of Kempe swaps with which it
can always reach a colorable board (each colorable board has depth 0).
For each configuration we list the total number of boards, and also a vector,
indexed from 0, where the $i$th coordinate is the number of boards of depth $i$.

\begin{thm}
If $G$ is an edge-critical graph with maximum degree $4$, then $G$ has average
degree at least $3.6$.  This is best possible, as shown by $K_5-e$.
\end{thm}
\begin{proof}
We use discharging with initial charge $\ch(v)=d(v)$ and the following rules.
\begin{enumerate}
\item[(R1)] Each 2-vertex takes $.8$ from each 4-neighbor.
\item[(R2)] Each 3-vertex with three 4-neighbors takes $.2$ from each 4-neighbor.
Each 3-vertex with two 4-neighbors takes $.3$ from each 4-neighbor.
\item[(R3)] Each 4-vertex with charge in excess of $3.6$ after (R2) splits this
excess evenly among its 4-neighbors with charge less than $3.6$.
\end{enumerate}
	
We write $\ch^*(v)$ for the final charge of vertex $v$.
We must show that every vertex $v$ finishes with $\ch^*(v)\ge 3.6$.

By VAL, each neigbhor of a 2-vertex $v$ is a 4-neighbor.  Thus,
$\ch^*(v)=2+2(.8)=3.6$.

By VAL, each 3-vertex $v$ has at least two 4-neighbors, so $\ch^*(v)=3+3(.2)$ or
$\ch^*(v)=3+2(.3)$.  In either case, $\ch^*(v)=3.6$.
	
	Now we consider a 4-vertex $v$.  Note that (R3) will never drop the charge of a
	4-vertex below $3.6$; thus, in showing that $\ch^*(v)\ge 3.6$, we need not
	consider (R3).
	
	\begin{clm}
		Let $v$ be a 4-vertex with no 2-neighbor. 
		If $v$ is not on a triangle with degrees 3,3,4, then $v$ finishes (R2) with
		charge at least 3.6; otherwise $v$ finishes (R2) with charge 3.4.
		\label{clm1}
	\end{clm}
	If $v$ has at most one 3-neighbor, then $v$ finishes (R2) with charge at least 3.7.
	By VAL, $v$ has at most two 3-neighbors, so assume exactly two.  If each
	receives charge $.2$ from $v$, then $v$ finishes (R3) with charge 3.6, as
	desired.  Otherwise, some 3-neighbor of $v$ has its own 3-neighbor.  Now $G$ has
	a path with degrees 3,4,3,3, which is (C\ref{fig:C2}), and hence reducible.  
	Here we use that $v$ does not lie on a triangle with degrees 3,3,4.
	
	\begin{clm}
		Let $v$ be a 4-vertex.  If $v$ has only 4-neighbors, then $v$ splits its excess
		charge of .4 at most 2 ways in (R3) \emph{unless} every $3^-$-vertex within
		distance two of $v$ is a 2-vertex; in that case $v$ may split its excess at most
		3 ways in (R3).
		\label{clm2}
	\end{clm}
	From the previous claim, we see that a 4-vertex needs charge after (R2) only if
	it has a 2-neighbor or if it lies on a triangle with degrees 3,3,4.  To prove the
	claim, we consider a 4-vertex $v$ with 4-neighbors $v_1,\ldots, v_4$ such that
	at least three $v_i$ either have 2-neighbors or lie on triangles with degrees 3,3,4. 
	
	First suppose that at least two $v_i$ lie on triangles with degrees 3,3,4.
	If two of these triangles are vertex disjoint, then we have (C\ref{fig:C5});
	otherwise, we have (C\ref{fig:C6}).  So we conclude that at most one $v_i$ lies
	on a triangle with degrees 3,3,4.  Assume that we have exactly one.
	Now we have either (C\ref{fig:C7}) or (C\ref{fig:C8}).  Thus, we conclude that
	no $v_i$ lies on a triangle with degree 3,3,4.  If each $v_i$ has a
	2-neighbor, then we have (C\ref{fig:C9}), (C\ref{fig:C10}), or (C\ref{fig:C11}).
	Thus, the claim is true.
	
	\begin{clm}
		Every vertex other than a 4-vertex with a 2-neighbor finishes (R3) with at least 3.6.
		\label{clm3}
	\end{clm}
	By Claim~\ref{clm1}, we need only consider a 4-vertex $v$ on a triangle with
	degrees 3,3,4; call its 3-neighbors $u_1$ and $u_2$.
	Since $v$ finishes (R2) with charge $3.4$, we must show that in
	(R3) $v$ receives charge at least $.2$.  By Claim~\ref{clm2}, this is true if
	$v$ has any 4-neighbor with no 3-neighbors (note that its 4-neighbors cannot
	have 2-neighbors, since that yields (C\ref{fig:C12})).  Thus, we assume that each
	4-neighbor of $v$, call them $u_3$ and $u_4$, has a 3-neighbor.  If $u_3$ or
	$u_4$ has a 3-neighbor other than $u_1$ or $u_2$, then the configuration is
	(C\ref{fig:C3}), which is reducible.  Thus, we may assume that $u_3$ is adjacent
	to $u_1$ and $u_4$ is adjacent to $u_2$.  Hence, each of $u_3$ and $u_4$
	finishes (R2) with charge $.1$.  It suffices to show that
	all of this charge goes to $v$ in (R3).  Thus, we need only show that no other
	neighbor of $u_3$ or $u_4$ needs charge after (R2).  If it does, then we have
	the reducible configuration (C\ref{fig:C4}), where possibly the rightmost 3 is a
	2.  Thus, $v$ finishes with charge $4-2(.3)+2(.1)=3.6$, as desired.
	
	\begin{clm}
		Every 4-vertex on a triangle with degrees 2,4,4 finishes (R3) with at least 3.6.
		\label{clm4}
	\end{clm}
	Let $v$ be a 4-vertex on a triangle with degrees 2,4,4, and let $v_1$ and $v_2$
	be its 2-neighbor and 4-neighbor on the triangle.  Let $v_3$ and $v_4$ be its
	other neighbors.  By VAL, $d(v_3)=d(v_4)=4$.  We will show that each of $v_3$
	and $v_4$ has only 4-neighbors and that each gives at least .2 to $v$ in (R3).
	If $v_3$ or $v_4$ has a $3^-$-neighbor, then we have (C\ref{fig:C13}) or
	(C\ref{fig:C1}), which are
	reducible; thus, each of $v_3$ and $v_4$ has only 4-neighbors.  By
	Claim~\ref{clm2}, vertex $v_3$ splits its charge of .4 at most two ways (thus,
	giving $v$ at least $.2$) unless it gives charge to exactly three of its
	neighbors, each of which has a 2-neighbor.  If this is the case, then we have
	(C\ref{fig:C14}), (C\ref{fig:C15}), or (C\ref{fig:C16}), each of which is
	reducible.  Thus, $v_3$ splits its charge at most two ways, and so gives $v$
	charge at least $.2$.  By the same argument, $v_4$ gives $v$ charge at least
	$.2$.  Thus, $v$ finishes (R3) with charge at least $4-.8+2(.2)=3.6$.
	\bigskip
	
	Now all that remains to consider is a 4-vertex $v$ with a 2-neighbor and three
	4-neighbors $v_1$, $v_2$, $v_3$.  Further, we may assume that $v$ does not lie
	on a triangle with degrees 2,4,4.  Also, we may assume that each $v_i$ has no
	2-neighbor; since $v$ lies on no 2,4,4 triangle, this would yield a copy of
	(C\ref{fig:C1}), which is reducible.
	
	\begin{clm}
		If any $v_i$ has two or more 3-neighbors, then $v$ finishes with at least $3.6$.
		\label{clm5}
	\end{clm}
	
	Suppose that $v_1$ has two 3-neighbors (by VAL it can have no more).  First, we
	note that neither $v_2$ nor $v_3$ has a 3-neighbor.  If it's distinct from those
	of $v_1$, then we have (C\ref{fig:C23}); otherwise, we have (C\ref{fig:C24}).
	Thus, each of $v_2$ and $v_3$ finishes (R2) with excess charge $.4$.  So, it
	suffices to show that each of $v_2$ and $v_3$ splits its excess charge at most
	two ways.  By Claim~\ref{clm2}, this is true unless $v_2$ (say) splits its
	charge among three 4-neighbors, each of which has a 2-neighbor. If $v_2$ does
	so, then we have (C\ref{fig:C25}), (C\ref{fig:C26}), or (C\ref{fig:C27}), each
	of which is reducible.  Thus, $v$ finishes (R3) with charge at least
	$4-.8+2(.2)=3.6$.
	
	\begin{clm}
		If any $v_i$ has a 3-neighbor, which itself has a 3-neighbor, then $v$ finishes
		with at least $3.6$.
		\label{clm6}
	\end{clm}
	Assume that $v_1$ has a 3-neighbor, which itself has a 3-neighbor.
	First, we note that neither $v_2$ nor $v_3$ has a 3-neighbor.
	If so, then we have one of (C\ref{fig:C34}), (C\ref{fig:C35}), or
	(C\ref{fig:C36}).
	Thus, each of $v_2$ and $v_3$ finishes (R2) with excess charge $.4$.  So, it
	suffices to show that each of $v_2$ and $v_3$ splits its excess charge at most
	two ways.  By Claim~\ref{clm2}, this is true unless $v_2$ (say) splits its
	charge among three 4-neighbors, each of which has a 2-neighbor. 
	If $v_2$ does so, then we have 
	(C\ref{fig:C37}), (C\ref{fig:C38}), or (C\ref{fig:C39}), each
	of which is reducible.  
	Thus, $v$ finishes (R3) with charge at least $4-.8+2(.2)=3.6$.
	
	\begin{clm}
		If no $v_i$ has a 3-neighbor, then $v$ finishes with at least $3.6$.
		\label{clm7}
	\end{clm}
	Since each $v_i$ has only 4-neighbors, it finishes (R2) with excess charge $.4$.
	By Claim~\ref{clm2}, it splits this charge at most 3 ways.  Thus, $v$ finishes
	with charge at least $4-.8+3(.4/3)=3.6$.
	
	\begin{clm}
		If at least two vertices $v_i$ have 3-neighbors, then $v$ finishes with at least
		$3.6$.
	\end{clm}
	Assume that $v_1$ and $v_2$ each have 3-neighbors.  By Claim~\ref{clm5}, each
	has exactly one 3-neighbor.  By Claim~\ref{clm6}, the 3-neighbors of $v_1$ and
	$v_2$ do not themselves have 3-neighbors.  Thus, each of $v_1$ and $v_2$
	finishes (R2) with charge exactly $3.8$.  Hence, in (R3), each splits its
	excess charge of $.2$ among its 4-neighbors that need charge.  We show that in
	(R3) all of this charge goes to $v$.  Suppose, to the contrary, that $v_1$
	sends some of its charge elsewhere; this could be to (i) a $4$-vertex with a
	2-neighbor or (ii) a 4-vertex on a triangle with degrees 3,3,4.  In (i), we
	have one of (C\ref{fig:C17})--(C\ref{fig:C20}).  In (ii), we have one of
	(C\ref{fig:C21}) and (C\ref{fig:C22}).  All such configurations are reducible,
	which proves the claim.
	
	\begin{clm}
		If exactly one vertex $v_i$ has a 3-neighbor, then $v$ finishes with at least
		$3.6$.
	\end{clm}
	Assume that $v_1$ has exactly one 3-neighbor (and $v_2$ and $v_3$ have no
	3-neighbors), so $v_2$ and $v_3$ each finish (R2) with excess charge $.4$.
	It suffices to show that either $v_2$ and $v_3$ each give $v$ at least half of
	their charge in (R3) or else at least one of them gives $v$ all its charge in
	(R3); assume not.  By symmetry, we assume that $v_2$ splits its charge at least
	three ways in (R3) and $v_3$ splits its charge at least two ways.  If $v_3$
	gives charge to a 4-vertex with a 2-neighbor (in addition to $v$), then we have
	one of (C\ref{fig:C28})--(C\ref{fig:C32}).  So assume that $v_3$ gives charge to
	a 4-vertex on a triangle with degrees 3,3,4.  By Claim~\ref{clm6}, neither of
	these 3-vertices on the triangle is the 3-neighbor of $v_1$.  If we do not have
	(C\ref{fig:C29}) or (C\ref{fig:C30}), then we must have (C\ref{fig:C33}), which
	is reducible.  This proves the claim.
\end{proof}

\begin{figure}[htb]
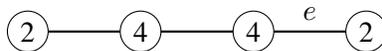

\renewcommand{\ttdefault}{ptm}
	\centering

	\caption{(C1)
		18 total boards: In increasing depths (14, 1, 1, 1, 1).\label{fig:C1}}
\end{figure}

\begin{figure}[htb]
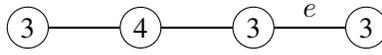

\renewcommand{\ttdefault}{ptm}
	\centering
	%
	\caption{(C2)
		26 total boards: In increasing depths (22, 1, 1, 1, 1).\label{fig:C2}}
\end{figure}

\begin{figure}[htb]
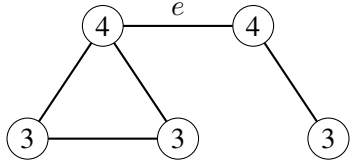

\renewcommand{\ttdefault}{ptm}
	\centering
	%
	\caption{ (C3) 84 total boards: In increasing depths (52, 20, 6, 4, 2).\label{fig:C3}}
\end{figure}

\begin{figure}[htb]
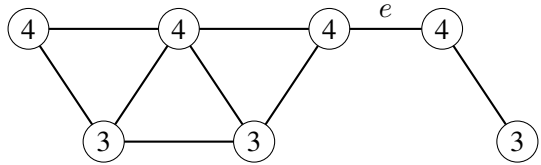

\renewcommand{\ttdefault}{ptm}
	\centering
	%
	\caption{ (C4) 32 total boards: In increasing depths (22, 5, 2, 2, 1).\label{fig:C4}}
\end{figure}

\begin{figure}[htb]
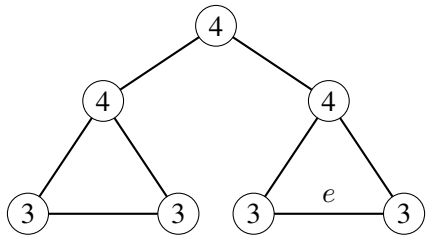

\renewcommand{\ttdefault}{ptm}
	\centering
	%
	\caption{ (C5) 652 total boards: In increasing depths (578, 36, 32, 6).\label{fig:C5}}
\end{figure}

\begin{figure}[htb]
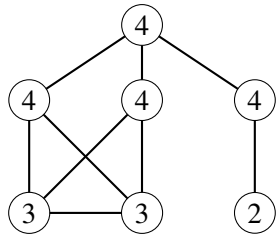

\renewcommand{\ttdefault}{ptm}
	\centering
	%
	\caption{ (C6)
		72 total boards: In increasing depths (49, 20, 3).  \label{fig:C6}}
\end{figure}

\begin{figure}[htb]
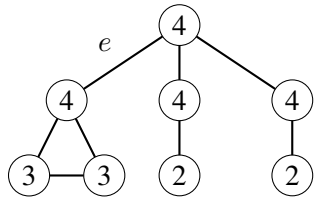

\renewcommand{\ttdefault}{ptm}
	\centering
	%
	\caption{(C7) 3936 total boards: In increasing depths 
		(2492, 750, 310, 266, 96, 14, 8).  \label{fig:C7}}
\end{figure}

\begin{figure}[htb]
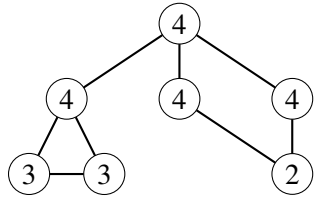

\renewcommand{\ttdefault}{ptm}
	\centering
	%
	\caption{ (C8) 391 total boards:
		In increasing depths (205, 163, 23).  \label{fig:C8}}
\end{figure}

\begin{figure}[htb]
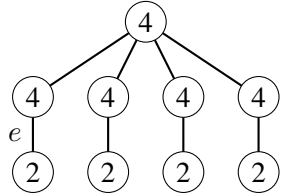

\renewcommand{\ttdefault}{ptm}
	\centering
	%
	\caption{ (C9)
		6546 total boards: In increasing depths (5277, 279, 282, 285, 186, 135, 81, 21).
		\label{fig:C9}}
\end{figure}

\begin{figure}[htb]
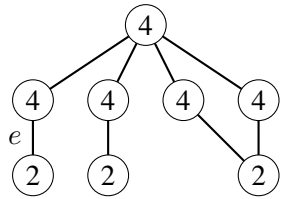

\renewcommand{\ttdefault}{ptm}
	\centering
	%
	\caption{ (C10)
		526 total boards: In increasing depths (429, 51, 28, 16, 2).
		\label{fig:C10}}
\end{figure}

\begin{figure}[htb]
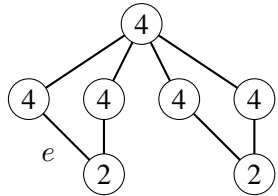

\renewcommand{\ttdefault}{ptm}
	\centering
	%
	\caption{ (C11)
		41 total boards: In increasing depths (36, 5).
		\label{fig:C11}}
\end{figure}

\begin{figure}[htb]
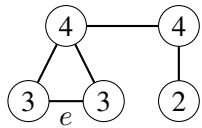

\renewcommand{\ttdefault}{ptm}
	\centering
	%
	\caption{ (C12)
		55 total boards: In increasing depths (49, 4, 2).
		\label{fig:C12}}
\end{figure}

\begin{figure}[htb]
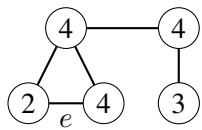

\renewcommand{\ttdefault}{ptm}
	\centering
	%
	\caption{ (C13)
		22 total boards: In increasing depths (18, 1, 1, 1, 1).
		\label{fig:C13}}
\end{figure}
\clearpage

\begin{figure}[htb]
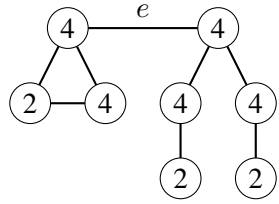

\renewcommand{\ttdefault}{ptm}
	\centering
	%
	\caption{(C14)
		1480 total boards: In increasing depths
		(868, 178, 140, 160, 92, 14, 18, 10).
		\label{fig:C14}}
\end{figure}

\begin{figure}[htb]
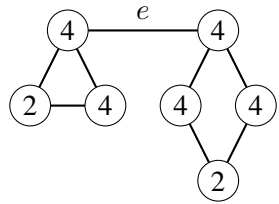

\renewcommand{\ttdefault}{ptm}
	\centering
	%
	\caption{(C15) 114 total boards: In increasing depths (74, 28, 12).\label{fig:C15}}
\end{figure}

\begin{figure}[htb]
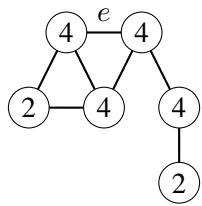

\renewcommand{\ttdefault}{ptm}
	\centering
	%
	\caption{(C16) 62 total boards: In increasing depths (44, 15, 3).
		\label{fig:C16}}
\end{figure}

\begin{figure}[htb]
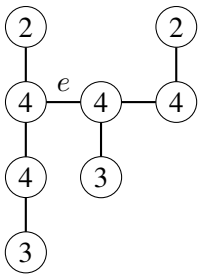

\renewcommand{\ttdefault}{ptm}
	\centering
	%
	\caption{(C17)
		7124 total boards
		In increasing depths (3648, 1081, 551, 579, 606, 448, 196, 15).\label{fig:C17}}
\end{figure}

\begin{figure}[htb]
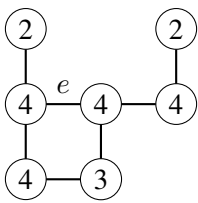

\renewcommand{\ttdefault}{ptm}
	\centering
	%
	\caption{(C18)
		1094 total boards 
		In increasing depths (570, 248, 131, 96, 43, 6).\label{fig:C18}}
\end{figure}

\begin{figure}[htb]
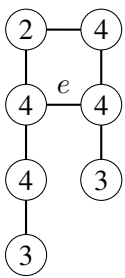

\renewcommand{\ttdefault}{ptm}
	\centering
	%
	\caption{(C19) 534 total boards
		In increasing depths (297, 162, 19, 15, 17, 9, 2, 2, 2, 2, 2, 2, 3).\label{fig:C19}}
\end{figure}

\begin{figure}[htb]
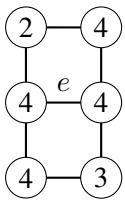

\renewcommand{\ttdefault}{ptm}
	\centering
	%
	\caption{(C20) 85 total boards
		In increasing depths: (47, 27, 4, 4, 3).\label{fig:C20}}
\end{figure}

\begin{figure}[htb]
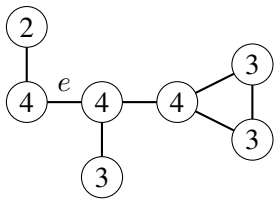

\renewcommand{\ttdefault}{ptm}
	\centering
	%
	\caption{(C21)
		1406 total boards:
		In increasing depths (790, 374, 188, 50, 4).\label{fig:C21}}
\end{figure}

\begin{figure}[htb]
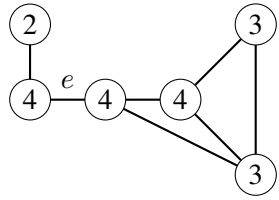

\renewcommand{\ttdefault}{ptm}
	\centering
	%
	\caption{(C22) 
		72 total boards: In increasing depths (57, 14, 1).\label{fig:C22}}
\end{figure}

\begin{figure}[htb]
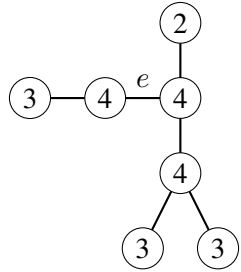

\renewcommand{\ttdefault}{ptm}
	\centering
%
	\caption{(C23)
		2230 total boards
		In increasing depths (1158, 426, 267, 225, 100, 48, 6).\label{fig:C23}}
\end{figure}

\begin{figure}[htb]
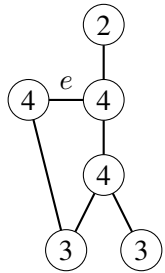

\renewcommand{\ttdefault}{ptm}
	\centering
	%
	\caption{(C24) 335 total boards
		In increasing depths (182, 96, 44, 10, 2, 1).\label{fig:C24}}
\end{figure}

\begin{figure}
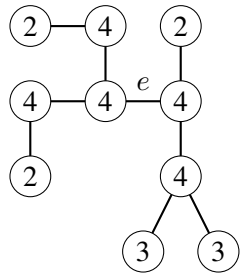

\renewcommand{\ttdefault}{ptm}
	\centering
	%
	\caption{(C25) 108416 total boards:
		In increasing depths (58218, 17840, 12265, 10729, 5935, 2981, 422, 26).\label{fig:C25}}
\end{figure}

\begin{figure}
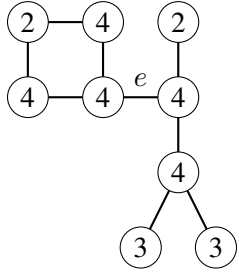

\renewcommand{\ttdefault}{ptm}
	\centering
	%
	\caption{(C26) 8022 total boards: 
		In increasing depths (4703, 1877, 1022, 420).\label{fig:C26}}
\end{figure}

\begin{figure}
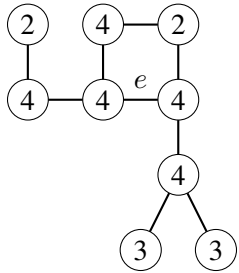

\renewcommand{\ttdefault}{ptm}
	\centering
	%
	\caption{(C27) 8239 total boards:
		In increasing depths (4666, 2373, 922, 270, 8).\label{fig:C27}}
\end{figure}

\begin{figure}
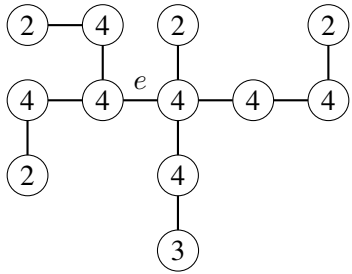

\renewcommand{\ttdefault}{ptm}
	\centering
	%
	\caption{(C28)
		579808 total boards:
		In increasing depths (292466, 64554, 28706, 24886, 24340, 23199, 23908, 26010,
		25574, 23929, 15613, 5823, 792, 8).\label{fig:C28}}
\end{figure}
\clearpage
\begin{figure}
\renewcommand{\ttdefault}{ptm}
	\centering
	%
	\caption{(C29) 1740 total boards:
		In increasing depths (929, 301, 132, 152, 146, 58, 18, 4).\label{fig:C29}}
\end{figure}

\begin{figure}
\renewcommand{\ttdefault}{ptm}
	\centering
	%
	\caption{(C30) 1786 total boards: In increasing depths 
		(925, 422, 107, 111, 99, 53, 24, 8, 6, 7, 7, 8, 9).\label{fig:C30}}
\end{figure}

\begin{figure}
\renewcommand{\ttdefault}{ptm}
	\centering
	%
	\caption{(C31) 
		47540 total boards: In increasing depths 
		(23949, 6503, 3553, 3446, 3507, 3148, 2287, 1007, 140).\label{fig:C31}}
\end{figure}

\begin{figure}
\renewcommand{\ttdefault}{ptm}
	\centering
	%
	\caption{(C32)
		46132 total boards: In increasing depths
		(23677, 8019, 3850, 3921, 3096, 2377, 998, 166, 28).\label{fig:C32}}
\end{figure}

\begin{figure}
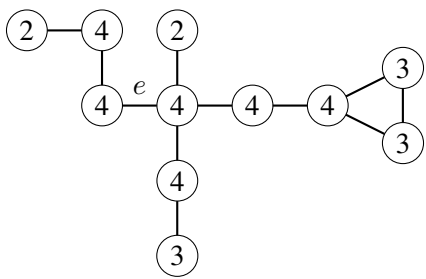

\renewcommand{\ttdefault}{ptm}
	\centering
	%
	
	\caption{(C33)
		134824 total boards: In increasing depths (66938, 19107, 11046, 9113, 8738,
		8349, 7140, 3636, 737, 20).\label{fig:C33}}
\end{figure}

\begin{figure}
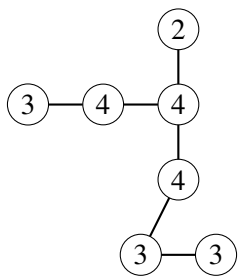

\renewcommand{\ttdefault}{ptm}
	\centering
	%
	\caption{(C34)
		3026 total boards:
		In increasing depth (1085, 380, 204, 147, 127, 143, 163, 145, 88, 90, 103, 28).
		(This configuration is unusual because there is no edge $e$ such that it
works to consider only near colorings for $G-e$; we need to go through a board
that is a near coloring for a different missing edge.)
		\label{fig:C34}}
\end{figure}

\begin{figure}
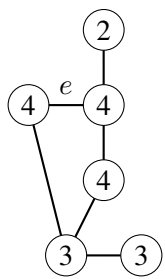

\renewcommand{\ttdefault}{ptm}
	\centering
	%
	\caption{(C35)
		142 total boards: In increasing depth (86, 49, 7).\label{fig:C35}}
\end{figure}

\begin{figure}
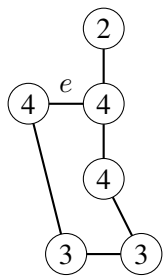

\renewcommand{\ttdefault}{ptm}
	\centering
	%
	\caption{(C36)
		328 total boards: In increasing depths (177, 50, 40, 41, 20).\label{fig:C36}}
\end{figure}

\begin{figure}
\renewcommand{\ttdefault}{ptm}
	\centering
	%
	\caption{(C37)
		108416 total boards: In increasing depths (54806, 12342, 6708, 6450, 6737, 6436,
		5975, 5109, 2990, 697, 114, 35, 17).\label{fig:C37}
	}
\end{figure}

\begin{figure}
\renewcommand{\ttdefault}{ptm}
	\centering
	
	%
	
	\caption{(C38)
		8022 total boards: In increasing depths (4449, 1573, 796, 790, 388,
		26).\label{fig:C38}
	}
\end{figure}

\begin{figure}
\renewcommand{\ttdefault}{ptm}
	\centering
	%
	\caption{(C39) 8239 total boards: In increasing depths
		(4424, 2063, 653, 628, 400, 69, 2).\label{fig:C39}
	}
\end{figure}
\clearpage

\end{document}